\documentclass[a4paper,10pt,psamsfonts]{amsart}
\usepackage[utf8]{inputenc} % direct use of umlauts
\usepackage[T1]{fontenc}    % font expansion for umlauts

\usepackage{amsmath}
\usepackage{amstext}
\usepackage{amsfonts}
\usepackage{amsthm}
\usepackage{amssymb}
\usepackage{graphicx}
\usepackage{psfrag}
\usepackage{subfigure}
\usepackage{float}

\usepackage{hyperref}
\hypersetup{colorlinks}
\hypersetup{linktocpage}

\usepackage{color}

\newcommand{\R}{{\mathbb{R}}}
\newcommand{\E}{{\mathbb{E}}}
\newcommand{\N}{{\mathbb{N}}}

\newcommand{\F}{{\mathcal{F}}} % for the sigma-algebra
\renewcommand{\P}{{\mathbf{P}}} % for the probability measure

\newcommand{\B}{{\mathcal{B}}}
\newcommand{\LB}{{\mathcal{L}}}
\newcommand{\diff}[1]{\,\mathrm{d}#1}

\newcommand{\id}{\mathrm{id}}
\newcommand{\ee}{\mathrm{e}}

\theoremstyle{plain}

\newtheorem{definition}{Definition}[section]
\newtheorem{theorem}[definition]{Theorem}
\newtheorem{lemma}[definition]{Lemma}
\newtheorem{corollary}[definition]{Corollary}
\newtheorem{prop}[definition]{Proposition}
\newtheorem{assumption}[definition]{Assumption}

\theoremstyle{definition}

\begin{document}

\title[Stochastic C-Stab. and B-cons. of Milstein-type
schemes] 
{Stochastic C-stability and B-consistency of explicit and implicit Milstein-type
schemes}

\author[W.-J.~Beyn]{Wolf-J\"urgen Beyn}
\address{Wolf-J\"urgen Beyn\\
Fakult\"at f\"ur Mathematik\\
Universit\"at Bielefeld\\
Postfach 100 131\\
DE-33501 Bielefeld\\
Germany}
\email{beyn@math.uni-bielefeld.de}

\author[E.~Isaak]{Elena Isaak}
\address{Elena Isaak\\
Fakult\"at f\"ur Mathematik\\
Universit\"at Bielefeld\\
Postfach 100 131\\
DE-33501 Bielefeld\\
Germany}
\email{eisaak@math.uni-bielefeld.de}

\author[R.~Kruse]{Raphael Kruse}
\address{Raphael Kruse\\
Technische Universit\"at Berlin\\
Institut f\"ur Mathematik, Secr. MA 5-3\\
Stra\ss e des 17.~Juni 136\\
DE-10623 Berlin\\
Germany}
\email{kruse@math.tu-berlin.de}

 % to be checked

\begin{abstract}
  This paper focuses on two variants of the Milstein scheme, 
  namely the split-step backward Milstein method and a newly proposed projected
  Milstein scheme, applied to stochastic differential equations which satisfy a
  global monotonicity condition. In particular, our assumptions include
  equations with super-linearly growing drift and diffusion coefficient
  functions and we show that both schemes are mean-square convergent of
  order $1$. Our analysis of the error of convergence with respect to the
  mean-square norm relies on the notion of stochastic
  C-stability and B-consistency, which was set up and applied to Euler-type
  schemes in [Beyn, Isaak, Kruse, J.~Sci.~Comp., 2015]. As a direct consequence
  we also obtain strong order $1$ convergence results for the split-step
  backward Euler method and the projected Euler-Maruyama scheme in the case of
  stochastic differential equations with additive noise. Our theoretical
  results are illustrated in a series of numerical experiments.
\end{abstract}

\maketitle
%\tableofcontents

\section{Introduction}
\label{sec:intro}

More than four decades ago Grigori N.~Milstein proposed 
a new numerical method for the approximate integration of stochastic ordinary
differential equations (SODEs) in \cite{milstein1974} (see \cite{milstein1975}
for an English translation). This scheme is nowadays called the \emph{Milstein
method} and offers a higher order of accuracy than the classical
Euler-Maruyama scheme. In fact, G.~N.~Milstein showed that his method converges
with order $1$ to the exact solution with respect to the root mean square norm
under suitable conditions on the coefficient functions of the SODE while the
Euler-Maruyama scheme is only convergent of order $\frac{1}{2}$, in
general.

In its simplest form, that is for scalar stochastic differential equations
driven by a scalar Wiener process $W$, the Milstein method is given by the
recursion 
\begin{align}
  \label{eq:Mil1d}
  \begin{split}
    X_h(t + h) &= X_h(t) + h f(t, X_h(t) )
    + g(t, X_h(t)) \Delta_h W(t) \\
    &\quad + \frac{1}{2} \big( \frac{\partial g}{\partial x}  \circ g \big) (t,
    X_{h} (t) ) \big( \Delta_h W(t)^2 - h \big),
  \end{split}
\end{align}
where $h$ denotes the step size, $\Delta_h W(t) = W(t + h) - W(t)$ is the
stochastic increment, and $f$ and $g$ are the drift and diffusion
coefficient functions of the underlying SODE (Equation \eqref{sode} below
shows the SODE in the full generality considered in this paper).

Since the
derivation of the Milstein method in
\cite{milstein1974} relies on an iterated 
application of the It\=o formula, the error analysis requires the 
boundedness and continuity of the coefficient functions $f$ and $g$ and their
partial derivatives up to the fourth order. Similar conditions also appear in 
the standard literature on this topic \cite{kloeden1999, milstein1995,
milstein2004}. 

In more recent publications these conditions have been
relaxed: For instance in \cite{kruse2011} it is proved that the strong
order $1$ result for the scheme \eqref{eq:Mil1d} stays true if the coefficient
functions are only two times continuously differentiable with bounded partial
derivatives, provided the exact solution has sufficiently high moments and the
mapping $x \mapsto (\frac{\partial g}{\partial x} \circ g) (t, x)$ is globally
Lipschitz continuous for every $t \in [0,T]$. On the other hand, from the
results in \cite{hutzenthaler2011} it follows that the explicit Euler-Maruyama
method is divergent in the strong and weak sense if the coefficient functions
are super-linearly growing. Since the same reasoning also applies to the
Milstein scheme \eqref{eq:Mil1d} it is necessary to consider suitable variants
in this situation.

%These conditions are still 
% 
%the standard assumptions in the literature (see also \cite{kloeden1999,
%milstein1995}) usually ask
%

One possibility to treat super-linearly growing coefficient functions is
proposed in \cite{wang2013}. Here the authors combine the
Milstein scheme with the taming strategy from \cite{hutzenthaler2012}. This
allows to prove the strong convergence rate $1$ in the case of SODEs 
whose drift coefficient functions satisfy a one-sided
Lipschitz condition. The same approach is used in
\cite{kumar2015}, where the authors consider SODEs driven by L\`evy noise.
However, both papers still require that the diffusion coefficient functions are
globally Lipschitz continuous. 

This is not needed for the implicit variant of the Milstein scheme considered
in \cite{higham2013}, where the strong convergence
result also applies to certain  SODEs with super-linearly growing
diffusion coefficient functions. However, the authors only consider scalar
SODEs and did not determine the order of convergence. The first result
bypassing all these restrictions is found in \cite{zhang2014}, which
deals with an explicit first order method based on a variant of the taming
idea. A more recent result based on the taming strategy is also given in
\cite{kumar2016}. 

In this paper we propose two further variants of the Milstein scheme
which apply to multi-dimensional SODEs of the form \eqref{sode}. First, we
follow an idea from \cite{beyn2015} and study the \emph{projected Milstein
method} which consists of the standard explicit Milstein scheme together with a
nonlinear projection onto a sphere whose radius
is expanding with a negative power of the step size. The second scheme is a
Milstein-type variant of the split-step backward Euler scheme (see
\cite{higham2002b}) termed \emph{split-step backward Milstein method}.

For both schemes we prove the optimal strong convergence rate $1$ in the
following sense: Let $X \colon [0,T] \times \Omega \to \R^d$ and
$X_h \colon \{t_0, t_1, \ldots,t_N\} \times \Omega \to \R^d$ denote
the exact solution and its numerical approximation with corresponding step
size $h$. Then, there exists a constant $C$ independent of $h$ such that
\begin{align}
  \label{eq0:strerr}
  \max_{n \in \{1,\ldots,N\} } \| X(t_n) - X_h(t_n) \|_{L^2(\Omega;\R^d)} \le C
  h,
\end{align}
where $t_n = n h$. For the proof we essentially impose the global monotonicity
condition \eqref{eq3:onesided} and certain local Lipschitz assumptions on the
first order derivatives of the drift and diffusion coefficient
functions. For a precise statement of all our assumptions and the two
convergence results we refer to Assumption~\ref{as:fg} and 
Theorems~\ref{th:PMilconv} and \ref{cor:SSBMconv} below. Together with the
result on the balanced scheme found in \cite{zhang2014}, these theorems are the
first results which determine the optimal strong convergence rate for some
Milstein-type schemes without any linear growth or global Lipschitz assumption
on the diffusion coefficient functions and for multi-dimensional SODEs.

The remainder of this paper is organized as follows: In Section~\ref{sec:prob}
we introduce the projected Milstein method and the split-step backward Milstein
scheme in full detail. We state all assumptions and the convergence
results, which are then proved in later sections. Further, we apply the
convergence results to SODEs with additive noise for which the Milstein-type
schemes coincide with the corresponding Euler-type scheme. 

The proofs follow the same steps as the error analysis 
in \cite{beyn2015}. In order to keep this paper as self-contained as possible
we briefly recall the notions of C-stability and B-consistency and the abstract
convergence theorem from \cite{beyn2015} in Section~\ref{sec:def}. Then, 
in the following four sections we verify that the two considered Milstein-type
schemes are indeed stable and consistent in the
sense of Section~\ref{sec:def}. Finally, in Section~\ref{sec:exp} we report on
a couple of numerical experiments which illustrate our theoretical findings.
Note that both examples include non-globally Lipschitz continuous coefficient
functions, which are not covered by the standard results found in
\cite{kloeden1999, milstein1995}.

\section{Assumptions and main results}
\label{sec:prob}
This section contains a detailed description of our assumptions on the
stochastic differential equation, under which our strong convergence results
hold. Further, we introduce the projected Milstein method and the split-step
backward Milstein scheme and we state our main results.

Our starting point is the stochastic ordinary differential equation
\eqref{sode} below.
We apply the same notation as in \cite{beyn2015} and we fix $d,m \in \N$, $T
\in (0,\infty)$, and a filtered probability space $(\Omega, \F, (\F_t)_{t \in
[0,T]},\P)$ satisfying the usual conditions. By $X \colon
[0,T] \times \Omega \to \R^d$ we denote a solution to the SODE
\begin{align}
  \label{sode}
  \begin{split}
    \diff{X(t)} &= f(t,X(t)) \diff{t} + \sum_{r=1}^m g^r(t,X(t))
    \diff{W^r(t)},\quad t \in [0,T],  \\
    X(0)&=X_0.
  \end{split}
\end{align}
Here $f\colon [0,T] \times \R^d \to \R^d$ stands for the drift coefficient
function, while $g^r \colon [0,T]\times \R^d \to \R^d$, $r=1,\ldots,m$,
are the diffusion coefficient functions. By $W^r \colon [0,T] \times
\Omega \to \R$, $r = 1,\ldots,m$, we denote an independent family of
real-valued standard $(\F_t)_{t\in [0,T]}$-Brownian motions on
$(\Omega,\mathcal{F},\P)$. For a sufficiently large $p \in [2,\infty)$ the
initial condition $X_0$ is assumed to be an element of the space
$L^p(\Omega,\F_0,\P;\R^d)$. 

Let us fix some further notation: We write
$\langle \cdot, \cdot \rangle$ and $|\cdot|$ for the Euclidean inner product
and the Euclidean norm on $\R^d$, respectively.  
Further, we denote by $\LB(\R^d) = \LB(\R^d,\R^d)$ the set of all bounded
linear operators on $\R^d$ endowed with the matrix norm $| \cdot |_{\LB(\R^d)}$
induced by the Euclidean norm. For a sufficiently smooth mapping $f
\colon [0,T] \times \R^d \to \R^d$ and a given $t \in [0,T]$ we denote by
$\frac{\partial f}{\partial x} (t,x) \in \LB(\R^d)$ the Jacobian matrix of the
mapping $\R^d \ni x \mapsto f(t,x) \in \R^d$.

Having established this we formulate the conditions on the drift and the
diffusion coefficient functions:

\begin{assumption}
  \label{as:fg}
  The mappings $f \colon [0,T] \times \R^d \to \R^d$ and
  $g^r \colon [0,T] \times \R^d \to \R^d$, $r = 1,\ldots,m$, are
  continuously differentiable. Further, there exist $L \in (0,\infty)$
  and $\eta \in (\frac{1}{2},\infty)$ such that for all $t \in [0,T]$ and
  $x_1,x_2 \in \R^d$ it holds
  \begin{align}
    \label{eq3:onesided}
    \big\langle f(t,x_1) - f(t,x_2), x_1-x_2 \big\rangle + \eta \sum_{r =
    1}^m  \big| g^r(t,x_1) - g^r(t,x_2) \big|^2 &\le L | x_1 - x_2 |^2.
  \end{align}
  In addition, there exists $q \in [2,\infty)$ such that 
  \begin{align}
    \big| \tfrac{\partial f}{\partial x} (t,x_1) - \tfrac{\partial f}{\partial
    x} (t,x_2) \big|_{\mathcal{L}(\R^d)} \le L \big( 1 +
    |x_1| + |x_2 | \big)^{q-2}  | x_1 - x_2 | 
    \label{eq3:loc_lip_f_x}
  \end{align}
  and, for every $r = 1,\ldots,m$,
  \begin{align}
    \big| \tfrac{\partial g^r}{\partial x} (t,x_1) - \tfrac{\partial
    g^r}{\partial x} (t,x_2)\big|_{\mathcal{L}(\R^d)} &\le L \big( 1 +
    |x_1| + |x_2 | \big)^{\frac{q-3}{2}}  | x_1 - x_2 |
    \label{eq3:loc_Lip_g_x}
  \end{align}
  for all $t \in [0,T]$ and $x_1,x_2 \in \R^d$. Moreover, it holds
  \begin{align}
    \big| \tfrac{\partial f}{\partial t} (t,x) \big| \le L \big( 1 + |x |
    \big)^q, \quad 
    \big| \tfrac{\partial g^r}{\partial t} (t,x) \big| &\le L \big( 1
 + |x|\big)^{\frac{q+1}{2}} , \label{eq3:poly_growth_t}
  \end{align}    
  for all $t \in [0,T]$, $x \in \R^d$, and all $r = 1,\ldots,m$.
\end{assumption}

First we note that Assumption~\ref{as:fg} is slightly weaker than the
conditions imposed in \cite[Lemma~4.2]{zhang2014} in terms of smoothness
requirements on the coefficient functions. Further, we recall that Equation
\eqref{eq3:onesided} is often termed \emph{global monotonicity condition} in
the literature. It is easy to check that Assumption~\ref{as:fg} 
is satisfied (with $q=3$) if $f$ and $g^r$ and all their first order partial
derivatives are globally Lipschitz continuous. However, Assumption~\ref{as:fg}
includes several SODEs which cannot be treated by the standard results found in
\cite{kloeden1999,milstein1995}. We refer to Section~\ref{sec:exp} for
two more concrete examples.

For a possibly enlarged $L$ the following estimates are an immediate
consequence of Assumption~\ref{as:fg} and the mean value theorem: For all
$t,t_1,t_2 \in [0,T]$ and $x,x_1,x_2 \in \R^d$ it holds
\begin{align}
  |f(t,x) | &\le L \big( 1 + |x | \big)^q,
  \label{eq3:poly_growth}\\
  \big| \tfrac{\partial f}{\partial x} (t,x) \big|_{\mathcal{L}(\R^d)} 
  &\le L \big( 1 + |x | \big)^{q-1},\label{eq3:poly_growth_x}\\
  | f(t_1,x) - f(t_2,x) | &\le L \big( 1 +
  |x|\big)^q   |t_1 - t_2|,\label{eq3:loc_Lip_t}\\
  | f(t,x_1) - f(t,x_2) | &\le L \big( 1 +
  |x_1| + |x_2 | \big)^{q-1}  | x_1 - x_2 |,
  \label{eq3:loc_Lip}
\end{align}
and, for all $r=1,\ldots,m$, 
\begin{align}
  | g^r(t,x)| &\le L \big( 1 + |x | \big)^{\frac{q+1}{2}},
  \label{eq3:poly_growth_g}\\
  \big| \tfrac{\partial g^r}{\partial x} (t,x) \big|_{\mathcal{L}(\R^d)} &\le
  L \big( 1 + |x | \big)^{\frac{q-1}{2}}, \label{eq3:poly_growth_g_x}\\
  | g^r(t_1,x) - g^r(t_2,x) | &\le L \big( 1 +
  |x| \big)^{\frac{q+1}{2}}  |t_1 - t_2|,\label{eq3:loc_Lip_g_t}\\
  | g^r(t,x_1) - g^r(t,x_2) | &\le L \big( 1 +
  |x_1| + |x_2 | \big)^{\frac{q-1}{2}}  | x_1 - x_2 |.
  \label{eq3:loc_Lip_g}
\end{align}    
Thus, Assumption~\ref{as:fg} implies \cite[Assumption~2.1]{beyn2015} and all
results of that paper also hold true in the situation considered here. Note that in this paper we use the weights $(1+|x|)^p$ instead of $1+|x|^p$
as in \cite{beyn2015}. For $p \ge 0$ this makes no difference, 
however in condition \eqref{eq3:loc_Lip_g_x} we may have $p=\frac{q-3}{2}<0$ if $2\le q <3$, so that Lipschitz constants actually decrease at infinity.

In the following it will be convenient to introduce the abbreviation
\begin{align}
  \label{eq3:grr}
  g^{r_1,r_2}(t,x) := 
  \frac{\partial g^{r_1}}{\partial x} (t,x) g^{r_2}(t,x), \quad t \in [0,T],\;
  x \in \R^d,
\end{align}
for $r_1,r_2 = 1,\ldots,m$. As above, one easily verifies under
Assumption~\ref{as:fg} that the mappings $g^{r_1,r_2}$ satisfy 
(for a possibly larger $L$) the polynomial
growth bound 
\begin{align}
  \label{eq3:grr_poly_growth}
  \big| g^{r_1,r_2}(t,x) \big| &\le L \big( 1 + |x| \big)^{q}  
\end{align}
as well as the local Lipschitz bound  
\begin{align}
  \label{eq3:grr_loc_Lip}
  \big| g^{r_1,r_2}(t,x_1) - g^{r_1,r_2}(t,x_2) \big| &\le L \big( 1 +
  |x_1| + |x_2 | \big)^{q-1}  | x_1 - x_2 |
\end{align}
for all $x,x_1,x_2 \in \R^d$, $t \in [0,T]$, and $r_1,r_2 = 1,\ldots,m$.  For
this conclusion to hold in case $q<3$, it is essential to  
use the modified weight function in \eqref{eq3:loc_Lip_g_x}.

We say that an almost surely continuous and $(\F_t)_{t \in [0,T]}$-adapted
stochastic process $X \colon [0,T] \times \Omega \to \R^d$ is a solution to
\eqref{sode} if it satisfies $\P$-almost
surely the integral equation
\begin{align}
  \label{exact}
  X(t) = X_0 + \int_{0}^{t} f(s,X(s)) \diff{s} + \sum_{r = 1}^m \int_{0}^t
  g^r(s,X(s)) \diff{W^r(s)}
\end{align}
for all $t \in [0,T]$. It is well-known that Assumption~\ref{as:fg}
is sufficient to ensure the existence of a unique solution to
\eqref{sode}, see for instance \cite{krylov1999}, \cite[Chap.~2.3]{mao1997} or 
the SODE chapter in
\cite[Chap.~3]{roeckner2007}. 

In addition, the exact solution has finite $p$-th moments, that is
\begin{align}
  \label{eq:moments}
  \sup_{t \in [0,T]} \big\| X(t) \big\|_{L^p(\Omega;\R^d)} < \infty,
\end{align}
if the following \emph{global coercivity condition} is satisfied: 
There exist $C \in (0,\infty)$ and $p \in [2, \infty)$ such that 
\begin{align}
  \label{eq:growthcond}
  \big\langle f(t,x), x \big\rangle + \frac{p-1}{2} \sum_{r =
  1}^m  \big| g^r(t,x) \big|^2 &\le C \big(1 + | x |^2 \big)
\end{align}
for all $x \in \R^d$, $t \in [0,T]$. A proof is found, for example, 
in \cite[Chap.~2.4]{mao1997}.  

For the formulation of the numerical methods we recall the following
terminology from \cite{beyn2015}: By $\overline{h} \in (0,T]$ 
we denote an \emph{upper step size bound}. Then, for every $N \in \N$ we say
that $h = (h_1,\ldots,h_{N}) \in (0,\overline{h}]^{N}$ is a \emph{vector of
(deterministic) step sizes} if $\sum_{i = 1}^N h_i = T$. Every vector of step
sizes $h$ induces a set of temporal grid points $\mathcal{T}_h$ given by
\begin{align*}
  \mathcal{T}_h := \Big\{ t_n := \sum_{i = 1}^n h_i \, : \, n = 0,\ldots,N
  \Big\},
\end{align*}
where $\sum_{\emptyset} = 0$. For short we write
\begin{align*}
  |h| := \max_{i \in \{ 1,\ldots,N\}} h_i
\end{align*}
for the \emph{maximal step size} in $h$.

Moreover, we recall from \cite{kloeden1999, milstein1995} the following
notation for the stochastic increments: Let $t,s \in [0,T]$ with $s < t$. Then
we define
\begin{align}
  \label{eq2:stochincr1}
  I_{(r)}^{s,t} := \int_s^t \diff{W^{r}(\tau)},
\end{align}
for $r \in \{1,\ldots,m\}$ and, similarly, 
\begin{align}
  \label{eq2:stochincr2}
  I_{(r_1,r_2)}^{s,t} := \int_s^t \int_s^{\tau_1} \diff{W^{r_1}(\tau_2)}
  \diff{W^{r_2}(\tau_1)},
\end{align}
where $r_1, r_2 \in \{1,\ldots,m\}$. 
%The iterated stochastic integrals
%$I_{(r_1,r_2)}^{t,s}$ are often called \emph{L\'evy areas} of the driving
%Wiener process and not easily generated on a computer. 
The joint family of the iterated stochastic integrals
$(I_{(r_1,r_2)}^{s,t})_{r_1,r_2 = 1}^m$ is not easily generated on a computer.
Besides special cases such as commutative noise one relies on an additional
approximation method from e.g.\ \cite{gaines1994, ryden2001, wiktorsson2001}.
We also refer to the corresponding discussion in
\cite[Chap.~10.3]{kloeden1999}.  
%\raphael{In Section ??? we will investigate how the approximation of
%$(I_{(r_1,r_2)}^{s,t})_{r_1,r_2 = 1}^m$ affects the strong error of the
%numerical methods. }

The first numerical scheme, which we study in this paper, is an  
explicit one-step scheme and termed \emph{projected
Milstein method} (PMil). It is the Milstein-type counterpart of the projected
Euler-Maruyama method form \cite{beyn2015} and consists of the standard Milstein
scheme and a projection onto a ball in $\R^d$ whose radius is expanding with 
a negative power of the step size.

To be more precise, let $h \in (0,\overline{h}]^N$, $N \in \N$, be an arbitrary
vector of step sizes with upper step size bound $\overline{h} = 1$. 
For a given parameter $\alpha \in (0,\infty)$ the PMil method is determined by
the recursion 
\begin{align}
  \label{eq:PMildef}
  \begin{split}
    \overline{X}_h^{\mathrm{PMil}}(t_i) &= \min\big( 1 , h_i^{-\alpha} \big|
    X_h^{\mathrm{PMil}}(t_{i-1}) \big|^{-1} \big)
    X_h^{\mathrm{PMil}}(t_{i-1}),\\
    X_h^{\mathrm{PMil}}(t_i) &= \overline{X}_h^{\mathrm{PMil}}(t_{i})
    + h_i f(t_{i-1}, \overline{X}_h^{\mathrm{PMil}}(t_{i}) )
    + \sum_{r = 1}^m g^r(t_{i-1}, \overline{X}_h^{\mathrm{PMil}}(t_{i}))
    I_{(r)}^{t_{i-1},t_{i}}\\
    &\quad  + \sum_{r_1, r_2 = 1}^m g^{r_1,r_2} (t_{i-1},
    \overline{X}_{h}^{\mathrm{PMil}}(t_i)) 
    I_{(r_2,r_1)}^{t_{i-1},t_{i}}, \quad \text{ for } 1 \le i \le N,
  \end{split}
\end{align}
where $X_h^{\mathrm{PMil}}(0) := X_0$. 
The results of Section~\ref{sec:PMilstab} indicate that the parameter value for
$\alpha$ is optimally chosen by setting $\alpha = \frac{1}{2(q-1)}$ 
in dependence of the growth rate $q$ appearing in Assumption~\ref{as:fg}. 
One aim of this paper is the proof of the following
strong convergence result for the PMil method. It follows directly from
Theorems~\ref{th:PMilstab} and \ref{th:PMilcons} as well as
Theorem~\ref{th:Bconv}.  

\begin{theorem}
  \label{th:PMilconv}
  Let Assumption~\ref{as:fg} be satisfied with polynomial growth rate $q \in
  [2,\infty)$. If the exact solution $X$ to \eqref{sode} satisfies $\sup_{\tau
  \in [0,T]} \| X(\tau) \|_{L^{8q-6}(\Omega;\R^d)} < \infty$, then the
  projected Milstein method \eqref{eq:PMildef} with parameter value $\alpha =
  \frac{1}{2(q-1)}$ and with arbitrary upper step size bound $\overline{h} \in
  (0,1]$ is strongly convergent of order $\gamma = 1$. 
\end{theorem}

%The strong error analysis of the PMil method is carried out in
%Section~\ref{sec:PMilstab} and Section~\ref{sec:PMilcons}. 

Next, we come to the second numerical scheme, which is called
\emph{split-step backward Milstein method} (SSBM). For a suitable upper step
size bound $\overline{h} \in (0,T]$ and a given vector of step
sizes $h = (h_1,\ldots,h_N) \in (0,\overline{h}]^N$, $N \in \N$, this method is
defined by setting $X_h^{\mathrm{SSBM}}(0) = X_0$ and by the recursion 
\begin{align}
  \label{eq:SSBMdef}
  \begin{split}
    \overline{X}_h^{\mathrm{SSBM}}(t_i) &= X_h^{\mathrm{SSBM}}(t_{i-1}) + h_i
    f(t_{i}, \overline{X}_h^{\mathrm{SSBM}}(t_i)),\\
    X_h^{\mathrm{SSBM}}(t_{i}) &= \overline{X}_h^{\mathrm{SSBM}}(t_i) + \sum_{r
    = 1}^m g^r(t_{i}, \overline{X}_{h}^{\mathrm{SSBM}}(t_i))
    I_{(r)}^{t_{i-1},t_{i}}\\
    &\quad + \sum_{r_1, r_2 = 1}^m g^{r_1,r_2}(t_{i},
    \overline{X}_{h}^{\mathrm{SSBM}}(t_i)) I_{(r_2,r_1)}^{t_{i-1},t_{i}},  
  \end{split}
\end{align}
for every $i = 1,\ldots,N$.

Let us note that the recursion defining the SSBM method evaluates the
diffusion coefficient functions $g^r$ at time $t_{i}$ in the $i$-th step.
This phenomenon was already apparent in the definition of the split-step
backward Euler method in \cite{beyn2015}. It turns out that by this
modification we avoid some technical issues in the proofs as condition
\eqref{eq2:onesidedSSBM} is applied to $f$ and $g^r$, $r =
1,\ldots,m$, simultaneously at the same point $t \in [0,T]$ in time. Compare
also with the inequality \eqref{eq:stab} further below.

It is shown in Section~\ref{sec:SSBMstab} that the
SSBM scheme is a well-defined stochastic one-step method under
Assumption~\ref{as:fg}. The second main result of this paper is the proof
of the following strong convergence result:

\begin{theorem}
  \label{cor:SSBMconv}
  Let Assumption~\ref{as:fg} be satisfied with $L \in (0,
  \infty)$ and $q \in [2,\infty)$. In addition, we assume that there exist
  $\eta_1 \in (1,\infty)$ and $\eta_2 \in (0,\infty)$ such that it holds 
  \begin{align}
    \label{eq2:onesidedSSBM}
    \begin{split}
      &\big\langle f(t,x_1) - f(t,x_2), x_1-x_2 \big\rangle + \eta_1 \sum_{r =
      1}^m  \big| g^r(t,x_1) - g^r(t,x_2) \big|^2 \\
      &\quad + \eta_2 \sum_{r_1,r_2 =
      1}^m  \big| g^{r_1,r_2}(t,x_1) - g^{r_1,r_2}(t,x_2) \big|^2
      \le L | x_1 - x_2 |^2
    \end{split}
  \end{align}
  for all $x_1, x_2 \in \R^d$. If the solution $X$ to
  \eqref{sode} satisfies $\sup_{\tau \in [0,T]} \| X(\tau)
  \|_{L^{6q-4}(\Omega;\R^d)} < \infty$, then the split-step backward Milstein
  method \eqref{eq:SSBMdef} with arbitrary upper step size bound
  $\overline{h} \in (0, \max(L^{-1},\frac{2 \eta_2}{\eta_1}))$
  is strongly convergent of order $\gamma = 1$. 
\end{theorem}

As we show below this theorem follows directly from Theorem~\ref{th:Bconv}
together with Theorems~\ref{th:SSBMstab} and \ref{th:SSBMcons}. Note that
\eqref{eq2:onesidedSSBM} is more restrictive than the global monotonicity
condition \eqref{eq3:onesided} if the mappings $g^{r_1,r_2}$ are not globally
Lipschitz continuous for all $r_1,r_2 = 1,\ldots,m$.

In the remainder of this section we briefly
summarize the corresponding convergence results in the case of stochastic
differential equations with \emph{additive noise}, that is
if the mappings $g^r$, $r = 1,\ldots,m$, do not depend explicitly on
the state of $X$.  In this case 
it is well-known that Milstein-type schemes coincide with their Euler-type
counterparts.

To be more precise, we consider the solution $X \colon [0,T]
\times \Omega \to \R^d$ to an SODE of the form
\begin{align}
  \label{sode_add}
  \begin{split}
    \diff{X(t)} &= f(t,X(t)) \diff{t} + \sum_{r=1}^m g^r(t)
    \diff{W^r(t)},\quad t \in [0,T],  \\
    X(0)&=X_0.
  \end{split}
\end{align}
In this case, the conditions on the drift coefficient function $f\colon [0,T]
\times \R^d \to \R^d$ and the diffusion coefficient functions $g^r \colon [0,T]
\to \R^d$, $r=1,\ldots,m$, in Assumption~\ref{as:fg} simplify to

\begin{assumption}[Additive noise]
  \label{as:fg_add}
  The coefficient functions $f \colon [0,T] \times \R^d \to \R^d$ and
  $g^r \colon [0,T] \to \R^d$, $r = 1,\ldots,m$, are
  continuously differentiable, and there exist constants $L \in (0,\infty)$,
$q \in [2,\infty)$
  such that for all $t \in [0,T]$ and
  $x,x_1,x_2 \in \R^d$ the following properties hold
  \begin{align*}
    \big\langle f(t,x_1) - f(t,x_2), x_1-x_2 \big\rangle &\le L | x_1 - x_2
    |^2,\\
    \big| \tfrac{\partial f}{\partial t} (t,x) \big| &\le L
    \big( 1 + |x | \big)^q,\\ 
    \big| \tfrac{\partial f}{\partial x} (t,x_1) - \tfrac{\partial f}{\partial
    x} (t,x_2) \big|_{\mathcal{L}(\R^d)} &\le L \big( 1 +
    |x_1| + |x_2 | \big)^{q-2} | x_1 - x_2 | .
  \end{align*}
\end{assumption}

Under this assumption it directly follows that $g^{r_1,r_2} \equiv 0$ for all
$r_1,r_2 = 1,\ldots,m$ for the coefficient functions
defined in \eqref{eq3:grr} . Consequently, the PMil method and the SSBM scheme
coincide with the PEM method and the SSBE scheme from \cite{beyn2015},
respectively. 

Let us note that Assumption~\ref{as:fg_add} implies the
global coercivity condition \eqref{eq:growthcond} for every $p \in [2,\infty)$.
Consequently, under Assumption~\ref{as:fg_add} the exact solution to
\eqref{sode_add} has finite $p$-th moments for every $p \in [2,\infty)$.
From this and Theorems~\ref{th:PMilconv} and \ref{cor:SSBMconv} we directly
obtain the following convergence result:

\begin{corollary}
  \label{cor:conv_add}
  Let Assumption~\ref{as:fg_add} be satisfied with $L \in (0,
  \infty)$ and $q \in [2,\infty)$. Then it holds that
  \begin{enumerate}
    \item[(i)] the projected Euler-Maruyama method 
      with $\alpha = \frac{1}{2(q-1)}$ and arbitrary upper step size bound
      $\overline{h} \in (0,1]$ is strongly convergent of order $\gamma 
      = 1$.  
    \item[(ii)] the split-step backward Euler method with
      arbitrary upper step size bound $\overline{h} \in (0, L^{-1})$ is
      strongly convergent of order $\gamma = 1$. 
  \end{enumerate}  
\end{corollary}

\section{A reminder on stochastic C-stability and B-consistency}
\label{sec:def}
In this section we give a brief overview of the notions of stochastic
C-stability and B-consistency introduced in \cite{beyn2015}.
We also state the abstract convergence theorem, which, roughly speaking, can be
summarized by 
\begin{align*}
  \text{stoch. C-stability } + \text{ stoch. B-consistency}
  \quad \Rightarrow   \text{ Strong convergence.}
\end{align*}

We first recall some additional notation from \cite{beyn2015}: 
For an arbitrary upper step size bound $\overline{h} \in (0,T]$ we define the
set $\mathbb{T} := \mathbb{T}(\overline{h}) \subset [0,T) \times
(0,\overline{h}]$ to be 
\begin{align*}
  \mathbb{T}(\overline{h})  := \big\{ (t,\delta) \in [0,T) \times
  (0,\overline{h}] \, : \, t+\delta \le T \big\}. 
\end{align*}
Further, for a given vector of step sizes $h \in (0,\overline{h}]^N$, $N \in
\N$, we denote
by $\mathcal{G}^{2}(\mathcal{T}_h)$ the space of all adapted and square
integrable \emph{grid functions}, that is
\begin{align*}
  \mathcal{G}^{2}(\mathcal{T}_h) := \big\{ Z \colon \mathcal{T}_h \times
  \Omega \to \R^d\, : \, Z(t_n) \in L^2(\Omega,\F_{t_n},\P;\R^d)\text{ for all
  } n = 0,1,\ldots,N \big\}.
\end{align*}
The next definition describes the abstract class
of stochastic one-step methods which we consider in this section.

\begin{definition}
  \label{def:onestep}
  Let $\overline{h} \in (0,T]$ be an upper step size bound and $\Psi \colon
  \R^d \times \mathbb{T} \times \Omega \to \R^d$ be a mapping satisfying the
  following measurability and integrability condition: For every 
  $(t,\delta) \in \mathbb{T}$ and $Z \in L^2(\Omega,\F_{t},\P;\R^d)$ it holds
  \begin{align}
    \label{eq:Psicond}
    \Psi(Z,t,\delta) \in L^2(\Omega,\F_{t+\delta},\P;\R^d).
  \end{align}
  Then, for every vector of step sizes $h \in
  (0,\overline{h}]^N$, $N \in \N$, 
  we say that a grid function $X_h \in \mathcal{G}^2(\mathcal{T}_h)$ is
  generated by the \emph{stochastic one-step method} $(\Psi,\overline{h},\xi)$
  with initial condition $\xi \in L^2(\Omega,\F_{0},\P;\R^d)$ if 
  \begin{align}
    \label{eq:onestep2}
    \begin{split}
      X_h(t_i) &= \Psi(X_h(t_{i-1}), t_{i-1}, h_i), \quad 1 \le i \le N,\\
      X_h(t_0) &= \xi.
    \end{split}
  \end{align}
 %  We call $\Psi$ the \emph{procedure function} of the stochastic one-step
%   method.
We call $\Psi$ the \emph{one-step map} of the method.
\end{definition}

For the formulation of the next definition we denote by $\E [ Y | \F_t]$ the
conditional expectation of a random variable $Y \in L^1(\Omega;\R^d)$ with
respect to the sigma-field $\F_t$. Note
that if $Y$ is square integrable, then $\E [ Y | \F_t]$ coincides with the
orthogonal projection onto the closed subspace $L^2(\Omega, \F_t, \P; \R^d)$.
By $(\id - \E[ \cdot | \F_t])$ we denote the associated projector onto the
orthogonal complement. 

\begin{definition}
  \label{def:cstab}
  A stochastic one-step method $(\Psi,\overline{h},\xi)$ is called
  \emph{stochastically C-stable} (with respect to the norm in
  $L^2(\Omega;\R^d)$) if there exist a constant $C_{\mathrm{stab}}$ and a
  parameter value $\nu \in (1,\infty)$ such that for all $(t,\delta) \in
  \mathbb{T}$ and all random variables $Y, Z \in L^2(\Omega,\F_{t},\P;\R^d)$
  it holds  
  \begin{align}
    \label{eq:stab_cond1}
    \begin{split}
      &\big\| \E \big[ \Psi(Y,t,\delta) -
      \Psi(Z,t,\delta) | \F_{t} \big]
      \big\|_{L^2(\Omega;\R^d)}^2\\
      &\qquad + \nu \big\| \big( \id - \E [ \, \cdot \, |
      \F_{t} ] \big) \big( \Psi(Y,t,\delta) -
      \Psi(Z,t,\delta) \big) \big\|^2_{L^2(\Omega;\R^d)}\\
      &\quad\le \big(1 + C_{\mathrm{stab}} \delta \big)
      \big\| Y - Z \big\|_{L^2(\Omega;\R^d)}^2.    
    \end{split}
  \end{align}
\end{definition}

A first consequence of the notion of stochastic C-stability is the following
\emph{a priori estimate}: Let $(\Psi,\overline{h},\xi)$ be a  
stochastically C-stable one-step method. If there exists a constant $C_0$
such that for all $(t,\delta) \in \mathbb{T}$ it holds
\begin{align}
  \label{eq3:cond1}
  \big\| \E \big[ \Psi(0, t, \delta) | \F_{t} \big]
  \big\|_{L^2(\Omega;\R^d)} &\le C_0 \delta,\\
  \label{eq3:cond2}
  \big\| \big( \id - \E \big[ \, \cdot \,
  | \F_{t} \big] \big) \Psi(0, t, \delta)
  \big\|_{L^2(\Omega;\R^d)}&\le C_0 \delta^{\frac{1}{2}}, 
\end{align}
then there exists a positive constant $C$ with
\begin{align*}
   &\max_{n \in \{0,\ldots,N\}} \| X_h(t_n)  \|_{L^2(\Omega;\R^d)}
   \le \ee^{C T}
   \Big( \| \xi  \|_{L^2(\Omega;\R^d)}^2 
   + C C_0^2 T \Big)^{\frac{1}{2}},
\end{align*}
for every vector of step sizes $h \in (0,\overline{h}]^N$, $N \in \N$, 
where $X_h$ denotes the grid function generated by $(\Psi,\overline{h},\xi)$
with step sizes $h$. A proof for this result is found in
\cite[Cor.~3.6]{beyn2015}.

\begin{definition}
  \label{def:bcons}
  A stochastic one-step method $(\Psi,\overline{h},\xi)$ is called
  \emph{stochastically B-consistent} of order $\gamma > 0$ to \eqref{sode} if
  there exists a constant $C_{\mathrm{cons}}$  such that for every
  $(t,\delta) \in \mathbb{T}$ it holds
  \begin{align}
    \label{eq:cons_cond1}
    \begin{split}
      \big\| \E \big[ X(t+\delta) - \Psi(X(t),t,\delta) | \F_{t} \big] 
      \big\|_{L^2(\Omega;\R^d)} \le C_{\mathrm{cons}} \delta^{ \gamma + 1}
    \end{split}
  \end{align}
  and
  \begin{align}
    \label{eq:cons_cond2}
    \begin{split}
      \big\| \big( \id - \E [ \, \cdot \, | \F_{t} ] \big)
      \big( X(t + \delta) - \Psi(X(t),t,\delta) \big)
      \big\|_{L^2(\Omega;\R^d)} \le C_{\mathrm{cons}} \delta^{ \gamma +
      \frac{1}{2}},
    \end{split}
  \end{align}
  where $X \colon [0,T] \times \Omega \to \R^d$ denotes the exact solution to
  \eqref{sode}.
\end{definition}

Finally, it remains to give our definition of strong convergence.

\begin{definition}
  \label{def:conv}
  A stochastic one-step method $(\Psi,\overline{h},\xi)$
  \emph{converges} \emph{strongly} with order $\gamma > 0$ to the exact
  solution of \eqref{sode} if there exists a constant $C$ such that for every
  vector of step sizes $h \in (0,\overline{h}]^N$, $N \in \N$, it holds
  \begin{align*}
    \max_{n \in \{0,\ldots,N\}} \big\| X_h(t_n) - X(t_n)
    \big\|_{L^{2}(\Omega;\R^d)} \le C |h|^{\gamma}.
  \end{align*}
  Here $X$ denotes the exact solution to \eqref{sode} and $X_h \in
  \mathcal{G}^2(\mathcal{T}_h)$ is the grid function generated by
  $(\Psi,\overline{h},\xi)$ with step sizes $h \in (0,\overline{h}]^N$. 
\end{definition}

We close this section with the following abstract convergence theorem, which is
proved in \cite[Theorem~3.7]{beyn2015}. 

\begin{theorem}
  \label{th:Bconv}
  Let the stochastic one-step method $(\Psi,\overline{h},\xi)$ be
  stochastically C-stable and stochastically B-consistent of order $\gamma >
  0$. If $\xi = X_0$, then there exists a constant $C$ depending on
  $C_{\mathrm{stab}}$, $C_\mathrm{cons}$, $T$, $\overline{h}$, and $\nu$
  such that for every vector of step sizes $h \in (0,\overline{h}]^N$, $N \in
  \N$, it holds 
  \begin{align*}
    \max_{n \in \{0,\ldots,N\}} \big\| X(t_n) - X_h(t_n)
    \big\|_{L^2(\Omega;\R^d)} \le C |h|^{\gamma},  
  \end{align*}
  where $X$ denotes the exact solution to \eqref{sode} and $X_h$ the grid
  function generated by $(\Psi,\overline{h},\xi)$ with step sizes $h$.  In
  particular, $(\Psi,\overline{h},\xi)$ is strongly convergent of order
  $\gamma$.
\end{theorem}

\section{C-stability of the projected Milstein method}
\label{sec:PMilstab}
In this section we prove that the projected Milstein (PMil) method defined in
\eqref{eq:PMildef} is stochastically C-stable.  

Throughout this section we assume that Assumption~\ref{as:fg} is satisfied with
growth rate $q \in [2,\infty)$. First, we choose an arbitrary upper step size
bound $\overline{h} \in (0, 1]$ and a parameter value $\alpha \in (0,\infty)$. 
Later it will turn out to be optimal to set $\alpha = \frac{1}{2(q-1)}$ in
dependence of the growth $q$ in Assumption~\ref{as:fg}.

For the definition of the one-step map of the PMil method it is convenient to
introduce the following short hand notation: For every $\delta \in
(0,\overline{h}]$, we denote the projection of $x \in \R^d$ onto the ball of
radius $\delta^{-\alpha}$ by  
\begin{align}
  \label{eq:circ}
  x^\circ := \min(1, \delta^{-\alpha} |x|^{-1}) x.
\end{align}
Then, the one-step map $\Psi^{\mathrm{PMil}} 
\colon \R^d \times \mathbb{T} \times \Omega \to \R^d$ is given by
\begin{align}
  \label{eq:PsiPMil}
  \begin{split}
    \Psi^{\mathrm{PMil}}(x,t,\delta) &:= x^\circ 
    + \delta f(t, x^\circ ) + \sum_{r = 1}^m g^r(t, x^\circ)
    I_{(r)}^{t,t+\delta} + \sum_{r_1,r_2 = 1}^m g^{r_1,r_2}(t,x^\circ)
    I_{(r_2,r_1)}^{t,t+\delta}    
  \end{split}
\end{align}
for every $x \in \R^d$ and $(t,\delta) \in \mathbb{T}$. Recall
\eqref{eq2:stochincr1} and \eqref{eq2:stochincr2} for the definition of the
stochastic increments. 

First, we check that the PMil method is a stochastic
one-step method in the sense of Definition~\ref{def:onestep}. At the same time
we verify that the one step map satisfies conditions \eqref{eq3:cond1} and
\eqref{eq3:cond2}.  

\begin{prop}  
  \label{prop:PMil}
  Let the functions $f$ and $g^r$, $r = 1,\ldots,m$, satisfy
  Assumption~\ref{as:fg} with $L \in (0,\infty)$ and let $\overline{h} \in
  (0,1]$. For every initial value $\xi \in L^2(\Omega;\F_{0},\P;\R^d)$ and for
  every $\alpha \in (0,\infty)$ it holds that $(\Psi^{\mathrm{PMil}},
  \overline{h}, \xi)$ is a stochastic one-step method.  
  
  In addition, there exists a constant $C_0$ only depending on $L$ and
  $m$ such that 
  \begin{align}
    \label{eq:PMilzero1}
    \big\| \E \big[ \Psi^{\mathrm{PMil}}( 0, t,\delta) | \F_{t} \big]
    \big\|_{L^2(\Omega;\R^d)} &\le C_0 \delta,\\
    \label{eq:PMilzero2}
    \big\| \big( \id - \E [ \, \cdot\, | \F_{t} ] \big)
    \Psi^{\mathrm{PMil}}( 0, t,\delta) \big\|_{L^2(\Omega;\R^d)} & \le C_0
    \delta^{\frac{1}{2}}    
  \end{align}
  for all $(t,\delta) \in \mathbb{T}$.
\end{prop}

\begin{proof}
  We first verify that $\Psi^{\mathrm{PMil}}$ satisfies \eqref{eq:Psicond}. For
  this let us fix arbitrary $(t, \delta) \in \mathbb{T}$ and $Z \in
  L^2(\Omega,\F_t,\P;\R^d)$. Then, the continuity and boundedness of the
  mapping $\R^d \ni x \mapsto x^\circ = \min(1, \delta^{-\alpha} |x|^{-1} ) x
  \in \R^d$ yields 
  \begin{align*}
    Z^\circ \in L^\infty(\Omega,\F_t,\P;\R^d).
  \end{align*}
  Consequently, by the smoothness of the coefficient functions and by
  \eqref{eq3:poly_growth}, \eqref{eq3:poly_growth_g}, and
  \eqref{eq3:grr_poly_growth} it follows that 
  \begin{align*}
    f(t,Z^\circ),\, g^{r_1}(t,Z^\circ),\, g^{r_1,r_2}(t,Z^\circ) \in
    L^\infty(\Omega,\F_t,\P;\R^d) 
  \end{align*}
  for every $r_1,r_2 = 1,\ldots,m$. Therefore, $\Psi^{\mathrm{PMil}}(Z,t,\delta)
  \colon \Omega \to \R^d$ is an $\F_{t+\delta} / \B(\R^d)$-measurable random
  variable satisfying condition \eqref{eq:Psicond}.

  It remains to show \eqref{eq:PMilzero1} and \eqref{eq:PMilzero2}.
  From \eqref{eq3:poly_growth} we get immediately that
  \begin{align*}
    \big\| \E \big[ \Psi^{\mathrm{PMil}}( 0, t,\delta) | \F_{t} \big]
    \big\|_{L^2(\Omega;\R^d)} =
    \big| \delta f(t, 0) \big| \le  L \delta.
  \end{align*}
  Next, recall that the stochastic increments $(I_{(r)}^{t,t+\delta})_{r = 1}^m$
  and $(I_{(r_1,r_2)}^{t,t+\delta})_{r_1,r_2 = 1}^m$ are pairwise uncorrelated.
  Therefore, we obtain that
  \begin{align*}
    &\big\| \big( \id - \E [ \, \cdot\, | \F_{t} ] \big)
    \Psi^{\mathrm{PMil}}( 0, t,\delta) \big\|_{L^2(\Omega;\R^d)}^2\\
    &\quad = \Big\| \sum_{r = 1}^m g^r(t, 0 ) I_{(r)}^{t,t+\delta}
    + \sum_{r_1,r_2 = 1}^m g^{r_1,r_2}(t, 0 ) I_{(r_1,r_2)}^{t,t+\delta}
    \Big\|^2_{L^2(\Omega;\R^d)}\\
    &\quad = \delta \sum_{r = 1}^m \big| g^r(t, 0 ) \big|^2 + \frac{\delta^2}{2}
    \sum_{r_1,r_2 = 1}^m \big| g^{r_1,r_2}(t, 0 ) \big|^2 
    \le L^2 m \delta + \frac{1}{2}L^2 m^2 \delta^2,
  \end{align*}
  where the last step follows from \eqref{eq3:poly_growth_g} and
  \eqref{eq3:grr_poly_growth}. Since $\delta \le \overline{h} \le 1$ this
  verifies \eqref{eq:PMilzero2}.  
\end{proof}

The next result is concerned with the projection onto the ball of radius
$\delta^{-\alpha}$. The proof is found in \cite[Lem.~6.2]{beyn2015}.

\begin{lemma}
  \label{lem:circ}
  For every $\alpha \in (0,\infty)$ and $\delta \in (0,1]$ the
  mapping $\R^d \ni x \mapsto x^\circ \in \R^d$ defined in \eqref{eq:circ} is
  globally Lipschitz continuous with Lipschitz constant $1$. In particular, it
  holds 
  \begin{align}
    \label{eq:PMilLip}
    \big| x^\circ_1 - x^\circ_2 \big| \le \big| x_1- x_2 \big|
  \end{align}
  for all $x_1, x_2 \in \R^d$.
\end{lemma}

The following inequality \eqref{eq:PMilcstab1} follows from the global
monotonicity condition \eqref{eq3:onesided} and plays an import role in the 
stability analysis of the PMil method. The proof is given in
\cite[Lem.~6.3]{beyn2015}. 

\begin{lemma}
  \label{lem:PMil}
  Let the functions $f$ and $g^r$, $r
  = 1,\ldots,m$, satisfy Assumption~\ref{as:fg} with $L \in
  (0,\infty)$, $q \in [2,\infty)$, and $\eta \in (\frac{1}{2},\infty)$. 
  Consider the mapping $\R^d \ni x \mapsto x^\circ \in \R^d$ defined in
  \eqref{eq:circ} with $\alpha \in (0, \frac{1}{2(q-1)}]$ and $\delta \in
  (0,1]$. Then there exists a constant $C$ only depending on $L$ such that
  \begin{align}
    \label{eq:PMilcstab1}
    \begin{split}
      & \big| x^\circ_1 - x_2^\circ + \delta ( f(t,x_1^\circ) - f(t,x_2^\circ))
      \big|^2 + 2 \eta \delta \sum_{r = 1}^m \big| g^r(t,x_1^\circ) -
      g^r(t,x_2^\circ)) \big|^2\\
      &\qquad \le (1 + C \delta) | x_1 - x_2 |^2
    \end{split}
  \end{align}
  for all $x_1, x_2 \in \R^d$.
\end{lemma}

The next theorem verifies that the PMil method is stochastically C-stable.

\begin{theorem}
  \label{th:PMilstab}
  Let the functions $f$ and $g^r$, $r = 1,\ldots,m$, satisfy
  Assumption~\ref{as:fg} with $L \in (0,\infty)$, $q \in [2,\infty)$, and
  $\eta \in (\frac{1}{2},\infty)$. Further, let 
  $\overline{h} \in (0, 1]$. Then, for every $\xi \in
  L^2(\Omega,\F_0,\P;\R^d)$ the projected Milstein method
  $(\Psi^{\mathrm{PMil}},\overline{h},\xi)$ with $\alpha = \frac{1}{2(q-1)}$ is
  stochastically C-stable.
\end{theorem}

\begin{proof}
  Let $(t,\delta) \in \mathbb{T}$ and $Y, Z \in L^2(\Omega,\F_t,\P;\R^d)$ be
  arbitrary. By recalling \eqref{eq:PsiPMil} we obtain
  \begin{align*}
    \E \big[ \Psi^{\mathrm{PMil}}(Y,t,\delta) -
    \Psi^{\mathrm{PMil}}(Z,t,\delta) | \F_{t} \big]
     = Y^\circ + \delta f(t, Y^\circ) - ( Z^\circ + \delta f(t,Z^\circ))
  \end{align*}
  and
  \begin{align*}
    &\big( \id - \E [ \, \cdot \, | \F_{t} ] \big)
    \big( \Psi^{\mathrm{PMil}}(Y,t,\delta) -
    \Psi^{\mathrm{PMil}}(Z,t,\delta)  \big)\\
    &\quad = \sum_{r = 1}^m \big( g^r(t , Y^\circ) - 
    g^r(t , Z^\circ) \big) I_{(r)}^{t,t+\delta} + \sum_{r_1,r_2 = 1}^m \big(
    g^{r_1,r_2}(t , Y^\circ) - g^{r_1,r_2}(t , Z^\circ) \big)
    I_{(r_2,r_1)}^{t,t+\delta}.
  \end{align*}
  In order to verify \eqref{eq:stab_cond1} with $\nu = 2 \eta \in (1,\infty)$
  let us note that the stochastic increments are pairwise uncorrelated and
  independent of $Y^\circ$ and $Z^\circ$. Hence it follows 
  \begin{align*}
    &\big\| \E \big[ \Psi^{\mathrm{PMil}}(Y,t,\delta) -
    \Psi^{\mathrm{PMil}}(Z,t,\delta) | \F_{t} \big]
    \big\|^2_{L^2(\Omega;\R^d)}\\
    &\qquad + \nu \big\| \big( \id - \E [ \, \cdot \, | \F_{t} ] \big)
    \big( \Psi^{\mathrm{PMil}}(Y,t,\delta) -  \Psi^{\mathrm{PMil}}(Z,t,\delta)
    \big) \big\|^2_{L^2(\Omega;\R^d)}\\
    &\quad = \big\| Y^\circ + \delta f(t, Y^\circ) - ( Z^\circ + \delta
    f(t,Z^\circ)) \big\|_{L^2(\Omega;\R^d)}^2\\
    &\qquad + \nu \delta \sum_{r = 1}^m \big\|  g^r(t , Y^\circ) - 
      g^r(t , Z^\circ) \big\|^2_{L^2(\Omega;\R^d)}\\
      &\qquad + \frac{1}{2}\nu \delta^2
      \sum_{r_1,r_2 = 1}^m \big\|  g^{r_1,r_2}(t , Y^\circ) -  g^{r_1,r_2}(t ,
      Z^\circ) \big\|^2_{L^2(\Omega;\R^d)}. 
  \end{align*}
  An application of Lemma~\ref{lem:PMil} with $\nu = 2 \eta$ shows that the
  first two terms are dominated by
  \begin{align*}
    &\E \Big[ \big| Y^\circ + \delta f(t, Y^\circ) - ( Z^\circ + \delta
    f(t,Z^\circ)) \big|^2 +  \nu \delta \sum_{r = 1}^m \big|  g^r(t ,
    Y^\circ) - g^r(t , Z^\circ) \big|^2 \Big]\\
    &\quad\le (1 + C \delta) \big\| Y - Z \big\|^2_{L^2(\Omega;\R^d)}.
  \end{align*}
  In addition, applications of \eqref{eq3:grr_loc_Lip} and
  \eqref{eq:PMilLip} yield
 \begin{align*}
    \big| g^{r_1,r_2}(t,x_1^\circ) - g^{r_1,r_2}(t,x_2^\circ) \big| &\le L
    \big( 1 + |x_1^\circ| + |x_2^\circ|  \big)^{q-1} \big|x_1^\circ
    - x_2^\circ \big| \\
     &\le L \big( 1 + 2 \delta^{-\alpha} \big)^{q-1} \big|x_1 - x_2 \big|,
  \end{align*}
  where we made use of the fact that $|x_1^\circ|, |x_2^\circ| \le
  \delta^{-\alpha}$. Since $\alpha(q-1) = \frac{1}{2}$ and $\delta \in (0,1]$
  it follows 
$\delta^{\frac{1}{2}}( 1 + 2 \delta^{-\alpha} )^{q-1} \le 3^{q-1}$ and,
  therefore, 
  \begin{align*}
    \nu \delta^2
    \sum_{r_1,r_2 = 1}^m \big\|  g^{r_1,r_2}(t , Y^\circ) -  g^{r_1,r_2}(t ,
    Z^\circ) \big\|^2_{L^2(\Omega;\R^d)} \le 3^{2(q-1)} \nu m^2 L^2 \delta \big\|
    Y - Z \big\|^2_{L^2(\Omega;\R^d)}.
  \end{align*}
  This completes the proof.
\end{proof}

\section{B-consistency of the projected Milstein method}
\label{sec:PMilcons}
In this section we show that the PMil method is stochastically $B$-consistent
of order $\gamma = 1$. To be more precise, we prove the following result: 

\begin{theorem}
  \label{th:PMilcons}
  Let $f$ and $g^r$, $r = 1,\ldots,m$, satisfy
  Assumption~\ref{as:fg} with $L \in (0,\infty)$ and $q \in [2,\infty)$. Let
  $\overline{h} \in (0,1]$ be arbitrary.
  If the exact solution $X$ to \eqref{sode} satisfies $\sup_{\tau \in [0,T]} \|
  X(\tau) \|_{L^{8q-6}(\Omega;\R^d)} < \infty$, then the projected Milstein
  method $(\Psi^{\mathrm{PMil}},\overline{h},X_0)$ with $\alpha =
  \frac{1}{2(q-1)}$ is stochastically B-consistent of order $\gamma = 1$.  
\end{theorem}

In preparation for the proof of Theorem~\ref{th:PMilcons} we introduce several
more technical lemmas. The first is cited from \cite[Lemma~6.5]{beyn2015}. It
formalizes a method of proof already found in \cite[Theorem~2.2]{higham2002b}. 

\begin{lemma}
  \label{lem:PMilcons1}
  For arbitrary $\alpha \in (0,\infty)$ and $\delta \in (0,1]$ consider the
  mapping $\R^d \ni x \mapsto x^\circ \in \R^d$ defined in \eqref{eq:circ}.
  Let $L \in (0,\infty)$, $\kappa \in [1, \infty)$ and let $\varphi \colon \R^d
  \to \R^d$ be a measurable mapping which satisfies
  \begin{align*}
    |\varphi(x) | \le L \big(1 + |x|^\kappa \big)
  \end{align*}
  for all $x \in \R^d$. For some $p \in (2,\infty)$ let $Y \in
  L^{p\kappa}(\Omega;\R^d)$. Then there exists a constant $C$ only depending on
  $L$ and $p$ with   
  \begin{align*}
    \big\| \varphi(Y) - \varphi(Y^\circ) \big\|_{L^2(\Omega;\R^d)} \le C \big(
    1 +  \|Y\|_{L^{p\kappa}(\Omega;\R^d)}^{\frac{1}{2}p \kappa} \big)
    \delta^{\frac{1}{2} \alpha (p -2 ) \kappa}.
  \end{align*}
\end{lemma}

The proof of consistency also depends on the H\"older continuity of the exact
solution to \eqref{sode} with respect to the norm in $L^p(\Omega;\R^d)$ for
some $p \in [2,\infty)$. A proof is given in \cite[Proposition~5.4]{beyn2015}.

\begin{prop}
  \label{prop:Hoelder}
  Let $f$ and $g^r$, $r = 1,\ldots,m$, satisfy
  Assumption~\ref{as:fg} with $L \in (0,\infty)$ and $q \in [2,\infty)$.
  For every $p \in [2,\infty)$  there exists a constant $C=C(L,q,p)$ such
  that 
  \begin{align} \label{eq:hoelderest}
    \big\| X(t_1) - X(t_2) \big\|_{L^p(\Omega;\R^d)} \le C \big(1 +
    \sup_{t \in [0,T]} \|X(t)\|_{L^{pq}(\Omega;\R^d)}^q \big)
    | t_1 - t_2|^{\frac{1}{2}}
  \end{align}
  holds for all $t_1, t_2 \in [0,T]$ and for every solution $X$ of\eqref{sode}
    satisfying \newline $\sup_{t \in [0,T]}\|X(t)\|_{L^{pq}(\Omega;\R^d)} < \infty$.
\end{prop}

The next auxiliary result combines the H\"older estimates with growth functions.

\begin{prop}
  \label{prop:Hoeldergrowth}
 Let $q_1 \ge 0, q_2 >0$ and consider an $\R^d$-valued process $X(t),t\in [0,T]$
satisfying \eqref{eq:hoelderest} for $pq=2(q q_2+ q_1)$ and 
$ \sup_{t\in [0,T]} \|X(t)\|_{L^{pq}(\Omega;\R^d)} < \infty$.
Then there exists a constant $C$ such that for all $0 \le t_1\le t_2 \le T$  
\begin{equation}  
\begin{aligned} \label{eq:hoeldergrowthest}
  \|\big(1 +& |X(t_1)|+|X(t_2)|\big)^{q_1} \big| X(t_1) - X(t_2) \big|^{q_2}
\|_{L^2(\Omega;\R)}\\ \le &C \big(1 +
    \sup_{t \in [0,T]} \|X(t)\|_{L^{pq}(\Omega;\R^d)}^{\frac{pq}{2}} \big)
    | t_1 - t_2|^{\frac{q_2}{2}}.
  \end{aligned}
\end{equation}
  \end{prop}
\begin{proof} We apply a H\"older estimate with arbitrary
$\nu> 1$, $\nu'=\frac{\nu}{\nu-1}$ and use \eqref{eq:hoelderest},
\begin{align*}
\|\big(1 +& |X(t_1)|+|X(t_2)|\big)^{q_1} \big| X(t_1) - X(t_2) \big|^{q_2}
\|_{L^2(\Omega;\R)}\\
\le & C \|1+|X(t_1)|+ |X(t_2)| \|_{L^{2 \nu' q_1}(\Omega;\R)}^{q_1}
\| X(t_2)-X(t_1)\|_{L^{2 \nu q_2 }(\Omega;\R^d)}^{q_2}\\
\le & C\big(1+\sup_{t\in [0,T]}\|X(t) \|_{L^{2 \nu' q_1}(\Omega;\R^d)}^{q_1}\big)
\big( 1+\sup_{t\in [0,T]}\|X(t) \|_{L^{2\nu q q_2 }(\Omega;\R^d)}^{q q_2}\big)
|t_1 -t_2|^{\frac{q_2}{2}} .
\end{align*}
The norms are balanced if we choose $2 \nu'q_1 = 2 q \nu q_2$ which leads
to $\nu = 1 + \frac{q_1}{q q_2}$ and  $2 \nu q q_2= 2(q q_2 + q_1)$. This
shows our assertion for $q_1  >0$. In case $q_1=0$ it is enough to just
apply \eqref{eq:hoelderest} with $p=2 q_2$.
\end{proof}

The following lemma is quoted from \cite[Lemma~5.5]{beyn2015}.

\begin{lemma}
  \label{lem:cons1} 
  Let Assumption~\ref{as:fg} be satisfied by $f$ and $g^r$, $r = 1,\ldots,m$,
  with $L \in (0,\infty)$ and $q \in [2,\infty)$.
  Further, let the exact solution $X$ to the SODE \eqref{sode} 
  satisfy $\sup_{t \in [0,T]} \| X(t) \|_{L^{4q-2}(\Omega;\R^d)} <
  \infty$. Then, there exists a constant $C$ such that for all $t_1, t_2,s \in
  [0,T]$ with $0 \le t_1 \le s \le t_2 \le T$ it holds
  \begin{align*}
    &\int_{t_1}^{t_2} \big\| f(\tau,X(\tau)) - f(s,X(t_1))
    \big\|_{L^2(\Omega;\R^d)} \diff{\tau}\\
    & \quad \le C \big( 1 + \sup_{t \in [0,T]} \big\| X(t)
    \big\|_{L^{4q-2}(\Omega;\R^d)}^{2q-1} \big) |t_1 - t_2|^{\frac{3}{2}}.
  \end{align*}
\end{lemma}

The order of convergence indicated by Lemma~\ref{lem:cons1} can be increased if
we insert the conditional expectation with respect to the $\sigma$-field
$\F_{t_1}$: 

\begin{lemma}
  \label{lem:cons1b} 
  Let Assumption~\ref{as:fg} be satisfied by $f$ and $g^r$, $r = 1,\ldots,m$
  with $L \in (0,\infty)$ and $q \in [2,\infty)$.
  Further, let the exact solution $X$ to the SODE \eqref{sode} 
  satisfy $\sup_{t \in [0,T]} \| X(t) \|_{L^{6q-4}(\Omega;\R^d)} <
  \infty$. Then, there exists a constant $C$ such that for all $t_1, t_2,s \in
  [0,T]$ with $0 \le t_1 \le s \le t_2 \le T$ it holds
  \begin{align*}
    &\int_{t_1}^{t_2} \big\| \E\big[ f(\tau,X(\tau)) - f(s,X(t_1)) |
    \F_{t_1} \big] \big\|_{L^2(\Omega;\R^d)} \diff{\tau}\\
    & \quad \le C \big( 1 + \sup_{t \in [0,T]} \big\| X(t)
    \big\|_{L^{6q-4}(\Omega;\R^d)}^{3q-2} \big) |t_1 - t_2|^{2}.
  \end{align*}
\end{lemma}

\begin{proof}
  Since $\| \E[ Y | \F_{t_1}] \|_{L^2(\Omega;\R^d)} \le \| Y
  \|_{L^2(\Omega;\R^d)}$ for all $Y \in L^2(\Omega;\R^d)$ 
  the integrand is estimated by
  \begin{align*}
    &\big\| \E\big[ f(\tau,X(\tau)) - f(s,X(t_1)) | \F_{t_1} \big]
    \big\|_{L^2(\Omega;\R^d)}\\ 
    &\;\le \big\| f(\tau,X(\tau)) - f(s,X(\tau)) \big\|_{L^2(\Omega;\R^d)}
    + \big\| \E\big[ f(s,X(\tau)) -
    f(s,X(t_1)) | \F_{t_1} \big] \big\|_{L^2(\Omega;\R^d)}
  \end{align*}
  for every $\tau \in [t_1,t_2]$. From \eqref{eq3:loc_Lip_t} it follows that
  \begin{align}
    \label{eq:term1}
    \begin{split}
      \big\| f(\tau,X(\tau)) - f(s,X(\tau)) \big\|_{L^2(\Omega;\R^d)} &\le L
      \| 1 + | X(\tau)| \|_{L^{2q}(\Omega;\R)}^q  |\tau - s | \\
      &\le  C \big( 1 + \sup_{t \in [0,T]} \| X(t) \|_{L^{2q}(\Omega;\R^d)}^q
      \big) |t_2 - t_1 |, 
    \end{split}
  \end{align}
  which after integrating over $\tau$,   yields the
  desired estimate since $2q \le 6q-4$ for  $q \ge 1$. 

  Next, from the mean value theorem we obtain
  \begin{align*}
    f(s,X(\tau)) - f(s,X(t_1)) = \frac{\partial f}{\partial x}(s, X(t_1)) \big(
    X(\tau) - X(t_1) \big) + R_f,
  \end{align*}
  where the remainder term $R_f$ is given by
  \begin{align*}
    R_f= \int_{0}^{1} \Big( \frac{\partial f}{\partial x}\big(s,
    X(t_1) + \rho (X(\tau) - X(t_1)) \big) - \frac{\partial f}{\partial x}(s,
    X(t_1)) \Big) \diff{\rho}\, \big( X(\tau) - X(t_1) \big).
  \end{align*}
  Using the SODE \eqref{sode} we obtain
  \begin{align*}
    \E \Big[ \frac{\partial f}{\partial x}(s, X(t_1)) \big(
    X(\tau) - X(t_1) \big) \Big| \F_{t_1} \Big]  = \E \Big[
    \frac{\partial f}{\partial x}(s, X(t_1)) 
    \int_{t_1}^{\tau} f(\sigma,X(\sigma) )  \diff{\sigma} \Big| \F_{t_1} \Big].
  \end{align*}
  After taking the $L^2$-norm and inserting \eqref{eq3:poly_growth} and
  \eqref{eq3:poly_growth_x} we arrive at
  \begin{align}
    \label{eq:term2}
    \begin{split}
      &\Big\| \E \Big[ \frac{\partial f}{\partial x}(s, X(t_1)) \big(
      X(\tau) - X(t_1) \big) \Big| \F_{t_1} \Big] \Big\|_{L^2(\Omega;\R^d)}\\
      &\quad \le \int_{t_1}^{\tau} \big\| L ( 1 + |X(t_1)| \big)^{q-1} L \big(
      1 + |X(\sigma)| \big)^{q} \big\|_{L^2(\Omega;\R)} \diff{\sigma}\\
      &\quad\le C \big( 1 + \sup_{t \in [0,T]} \big\| X(t)
      \big\|_{L^{4q-2}(\Omega;\R^d)}^{2q-1} \big) |\tau - t_1|. 
    \end{split}
  \end{align}
  Hence, we also obtain the desired estimate for this term after integrating
  over $\tau$.

  Finally, we have to estimate the $L^2$-norm of the remainder term $R_f$. For
  this we make use of \eqref{eq3:loc_lip_f_x} and get
\begin{equation}\label{eq:remainderf}
  \begin{aligned}
    | R_f | &\le \int_0^1 L \big( 1 + | X(t_1) + \rho (X(\tau) - X(t_1))
    | + |  X(t_1) | \big)^{q-2} \diff{\rho}\, \big|  X(\tau) - X(t_1)
    \big|^2\\
    &\le C \big( 1 + \big| X(t_1) \big| + \big| X(\tau) \big| \big)^{q-2}
    \big|  X(\tau) - X(t_1) \big|^2
  \end{aligned}
\end{equation}
  for a constant $C$ only depending on $L$ and $q$.
 % Hence, if $q = 2$ then we
 %  get from Proposition~\ref{prop:Hoelder}
 %  \begin{align*}
 %    \| R_f \|_{L^2(\Omega;\R^d)} \le C \big\| X(\tau) - X(t_1)
 %    \big\|_{L^{4}(\Omega;\R^d)}^2 \le C \big( 1 + \sup_{t \in
 %    [0,T]}\|X(t)\|^2_{L^8(\Omega;\R^d)} \big)^2 |\tau - t_1|,
 %  \end{align*}
 %  from which one derives the asserted estimate after integrating over $\tau$.
 %  If $q \in (2,\infty)$ we apply the H\"older inequality with exponents $\nu =
 %  \frac{3q - 2}{2q} > 1$ and $\nu' = \frac{\nu}{\nu - 1} =
 %  \frac{3q - 2}{q-2}$ to \eqref{eq:remainderf} and obtain
 %  \begin{align*}
 %    &\| R_f \|_{L^2(\Omega;\R^d)} \le C \Big( \E \big[ \big( 1 + \big| X(t_1)
 %    \big|^{q-2} + \big| X(\tau) \big|^{q-2}\big)^2 \big|  X(\tau) - X(t_1)
 %    \big|^4 \big] \Big)^{\frac{1}{2}}\\
 %    &\quad \le C \big\| 1 + \big| X(t_1) \big|^{q-2} + \big| X(\tau)
 %    \big|^{q-2} \big\|_{L^{2\nu'}(\Omega;\R)}  \big\|  X(\tau) - X(t_1)
 %    \big\|^2_{L^{4\nu}(\Omega;\R^d)}\\
 %    &\quad \le C \big( 1 + \sup_{t \in [0,T]} \big\| X(t) \big\|^{q-2}_{L^{2
 %    \nu' (q-2)}(\Omega;\R^d)} \big) \big( 1 + \sup_{t \in
 %    [0,T]}\|X(t)\|^q_{L^{4 \nu q}(\Omega;\R^d)} \big)^2 |\tau - t_1|, 
 %  \end{align*}
 %  where we also applied  Proposition~\ref{prop:Hoelder} with $p = 4 \nu$ in the
 %  last step. From our choice of $\nu$ it follows that $4 \nu q = 2 \nu' (q-2) =
 %  6q -4$.
Applying Proposition \ref{prop:Hoeldergrowth} with $q_2=2, q_1=q-2$ yields
$2(q q_2 + q_1) = 6q-4$ and therefore,
  \begin{align}
    \label{eq5:Rf}
    \| R_f \|_{L^2(\Omega;\R^d)} \le C \big( 1 + \sup_{t \in [0,T]} \| X(t)
    \|^{3q - 2}_{L^{6q -4}(\Omega;\R^d)} \big) |\tau - t_1|.
  \end{align}
  \end{proof}

The next lemma contains the corresponding estimate for the stochastic integral.

\begin{lemma}
  \label{lem:cons2} 
  Let Assumption~\ref{as:fg} be satisfied by $f$ and $g^r$, $r = 1,\ldots,m$
  with $L \in (0,\infty)$ and $q \in [2,\infty)$. 
  Further, let the exact solution $X$ to \eqref{sode} 
  satisfy $\sup_{t \in [0,T]} \| X(t) \|_{L^{6q-4}(\Omega;\R^d)} <
  \infty$. Then, there exists a constant $C$ such that for all $r =1,\ldots,m$
  and $t_1, t_2,s \in [0,T]$ with $0 \le t_1 \le s \le t_2 \le T$ it holds
  \begin{align*}
    &\Big\| \int_{t_1}^{t_2} g^{r}(\tau,X(\tau)) -
    g^{r}(s,X(t_1)) \diff{W^{r}(\tau)} - \sum_{r_2 = 1}^m 
    g^{r,r_2}(s,X(t_1)) 
    I_{(r_2,r)}^{t_1,t_2}  \Big\|_{L^2(\Omega;\R^d)}\\ 
    & \quad \le C  \big( 1 + \sup_{t \in [0,T]} \big\| X(t)
    \big\|_{L^{6q-4}(\Omega;\R^d)}^{3q - 2} \big) |t_1 - t_2|^{\frac{3}{2}}.
  \end{align*}
\end{lemma}

\begin{proof}
  Let us fix $r =1,\ldots,m$ arbitrary. We first 
  consider the square of the $L^2$-norm and by recalling \eqref{eq2:stochincr1}
  and \eqref{eq2:stochincr2} we get
  \begin{align*}
    &\E \Big[ \Big|
    \int_{t_1}^{t_2} \Big( g^{r}(\tau,X(\tau)) -
    g^{r}(s,X(t_1)) - \sum_{r_2 = 1}^m g^{r,r_2}(s,X(t_1))
    I_{(r_2)}^{t_1,\tau}\Big) \diff{W^{r}(\tau)} \Big|^2 \Big]\\
    &\quad = \int_{t_1}^{t_2} \E \Big[ \Big| 
    g^{r}(\tau,X(\tau)) - g^{r}(s,X(t_1)) - \sum_{r_2 = 1}^m
    g^{r,r_2}(s,X(t_1)) I_{(r_2)}^{t_1,\tau} \Big|^2 \Big] \diff{\tau}
  \end{align*}
  by an application of the It\=o isometry. Thus, the assertion is proved if
  there exists a constant $C$ independent of $\tau$, $t_1$, $t_2$, and $s$ such
  that 
  \begin{align*}
    \Gamma(\tau) &:= \Big\| 
    g^{r}(\tau,X(\tau)) - g^{r}(s,X(t_1)) - \sum_{r_2 = 1}^m
    g^{r,r_2}(s,X(t_1)) 
    I_{(r_2)}^{t_1,\tau} \Big\|_{L^2(\Omega;\R^d)}\\
    &\le C \big( 1 + \sup_{t \in [0,T]} \big\| X(t)
    \big\|_{L^{6q-4}(\Omega;\R^d)}^{3q-2} \big) |t_1 - t_2|
  \end{align*}
  for every $\tau \in [t_1,t_2]$. For this we first
  estimate $\Gamma(\tau)$ by
  \begin{align*}
    \Gamma(\tau) &\le \big\| g^{r}(\tau,X(\tau)) - g^{r}(s,X(\tau)) 
    \big\|_{L^2(\Omega;\R^d)} \\
    &\quad + \Big\|  g^{r}(s,X(\tau))
    - g^{r}(s,X(t_1)) - \sum_{r_2 = 1}^m g^{r,r_2}(s,X(t_1)) 
    I_{(r_2)}^{t_1,\tau} \Big\|_{L^2(\Omega;\R^d)}.
  \end{align*}
  In the same way as in \eqref{eq:term1} one shows for the first term
  \begin{align*}
    \big\| g^{r}(\tau,X(\tau)) - g^{r}(s,X(\tau)) 
    \big\|_{L^2(\Omega;\R^d)} \le C \big( 1 + \sup_{t\in[0,T]} \| X(t)
    \|^{\frac{q+1}{2}}_{L^{q+1}(\Omega;\R^d)} \big) |t_2 - t_1 |
  \end{align*}
and notes $q+1 \le 6 q -4$.
  Next, we again apply the mean value theorem 
  \begin{align*}
    g^r(s,X(\tau)) - g^r(s,X(t_1)) = \frac{\partial g^r}{\partial x}(s, X(t_1))
    \big( X(\tau) - X(t_1) \big) + R_g,
  \end{align*}
  where this time the remainder term $R_g$ is given by
  \begin{align*}
    R_g:= \int_{0}^{1} \Big( \frac{\partial g^r}{\partial x}\big(s,
    X(t_1) + \rho (X(\tau) - X(t_1)) \big) - \frac{\partial g^r}{\partial x}(s,
    X(t_1)) \Big) \diff{\rho}\, \big( X(\tau) - X(t_1) \big).
  \end{align*}
  Using the condition \eqref{eq3:loc_Lip_g_x} we  get
  \begin{align*} 
    |R_g| \le C \big( 1+ |X(t_1)|+ |X(\tau)| \big)^{q_1} |X(\tau)-X(t_1)|^2,
    \quad \text{where} \;q_1= \tfrac{(q-3)_+}{2}.
  \end{align*}
  Therefore,  Proposition \ref{prop:Hoeldergrowth} applies with $q_2=2$ and
  leads to 
  \begin{align*}
    \| R_g \|_{L^2(\Omega;\R^d)} \le C \big( 1 + \sup_{t \in [0,T]} \| X(t)
    \|^{3q-2}_{L^{6q-4}(\Omega;\R^d)} \big) |\tau - t_1|,
  \end{align*}
  since $2(q q_2+q_1)=\max(5q-3,4q)\le 6q-4$ for $q\ge 2$. 
  It remains to give a corresponding estimate for
  \begin{align*}
    \Gamma_2(\tau) := \Big\| \frac{\partial g^r}{\partial x}(s, X(t_1))
    \big( X(\tau) - X(t_1) \big) - \sum_{r_2 = 1}^m
    g^{r,r_2}(s,X(t_1)) 
    I_{(r_2)}^{t_1,\tau} \Big\|_{L^2(\Omega;\R^d)}.
  \end{align*}
  After inserting \eqref{exact} we finally arrive at the two terms 
  \begin{align*}
    &\Gamma_2(\tau) \le \Big\| \frac{\partial g^r}{\partial x}(s, X(t_1)) 
    \int_{t_1}^{\tau} f(\sigma,X(\sigma)) \diff{\sigma}
    \Big\|_{L^2(\Omega;\R^d)}\\
    &\quad+ \sum_{r_2 = 1}^m \Big\| \frac{\partial g^r}{\partial x}(s, X(t_1)) 
    \int_{t_1}^{\tau} g^{r_2}(\sigma,X(\sigma)) \diff{W^{r_2}(\sigma)} - 
    g^{r,r_2}(s,X(t_1)) I_{(r_2)}^{t_1,\tau}
    \Big\|_{L^2(\Omega;\R^d)}. 
  \end{align*}
  Using \eqref{eq3:poly_growth_g_x}, the first term is estimated analogously to \eqref{eq:term2},
\begin{align*}
\Big\| \frac{\partial g^r}{\partial x}(s, X(t_1)) 
    \int_{t_1}^{\tau} f(\sigma,X(\sigma)) \diff{\sigma}
    \Big\|_{L^2(\Omega;\R^d)} \le C
    \big(1+\sup_{t\in[0,T]}\|X(t)\|_{L^{3q-1}(\Omega,\R^d)}^{\frac{3q-1}{2}}
    \big)|\tau -t_1|. 
\end{align*} 
For the second term
  we insert \eqref{eq3:grr} and \eqref{eq2:stochincr1} and obtain 
 from It\=o's isometry 
  \begin{align*}
    &\sum_{r_2 = 1}^m \Big\| \frac{\partial g^r}{\partial x}(s, X(t_1)) 
    \int_{t_1}^{\tau} g^{r_2}(\sigma,X(\sigma)) \diff{W^{r_2}(\sigma)} - 
    g^{r,r_2}(s,X(t_1)) I_{(r_2)}^{t_1,\tau}
    \Big\|_{L^2(\Omega;\R^d)}\\
    &\quad = \sum_{r_2 = 1}^m \Big\| \int_{t_1}^{\tau} \frac{\partial
    g^r}{\partial x}(s, X(t_1)) \big(g^{r_2}(\sigma,X(\sigma)) -
    g^{r_2}(s,X(t_1)) \big) \diff{W^{r_2}(\sigma)} \Big\|_{L^2(\Omega;\R^d)}\\
    &\quad = \sum_{r_2 = 1}^m \Big( \int_{t_1}^{\tau} \Big\|
    \frac{\partial g^r}{\partial x}(s, X(t_1)) \big(g^{r_2}(\sigma,X(\sigma)) -
    g^{r_2}(s,X(t_1)) \big) \Big\|_{L^2(\Omega;\R^d)}^2 \diff{\sigma}
    \Big)^{\frac{1}{2}}. 
  \end{align*}
  Now, it follows from \eqref{eq3:loc_Lip_g_t} and \eqref{eq3:loc_Lip_g} that 
  \begin{align*}
    &\big| g^{r_2}(\sigma,X(\sigma)) - g^{r_2}(s,X(t_1)) \big|\\
    &\quad \le \big|
    g^{r_2}(\sigma,X(\sigma)) - g^{r_2}(\sigma,X(t_1))  \big| + \big|
    g^{r_2}(\sigma,X(t_1)) - g^{r_2}(s,X(t_1)) \big|\\ 
    &\quad \le L \big( 1 + |X(t_1)| + |X(\sigma)|  \big)^{\frac{q-1}{2}}
    |X(\sigma) -X(t_1) | 
    + L \big( 1 + |X(t_1)| \big)^{\frac{q+1}{2}} |\sigma - s |.
  \end{align*}
  Hence, the growth estimate \eqref{eq3:poly_growth_g_x} and 
 Proposition \ref{prop:Hoeldergrowth} with $q_1=q-1,q_2=1$ yield
  \begin{align*}
    &\Big\| \frac{\partial g^r}{\partial x}(s, X(t_1))
    \big(g^{r_2}(\sigma,X(\sigma)) - g^{r_2}(s,X(t_1)) \big)
    \Big\|_{L^2(\Omega;\R^d)} \\
    &\quad \le L \big\|  \big(1 + |X(t_1)|)^{\frac{q-1}{2}}  \big|
    g^{r_2}(\sigma,X(\sigma)) - g^{r_2}(s,X(t_1)) \big|
    \big\|_{L^2(\Omega,\R)}\\ 
    &\quad \le L^2 \big\| \big(1 + |X(t_1)| + |X(\sigma)|\big)^{q-1}
     \big| X(\sigma) - X(t_1) \big|\big\|_{L^2(\Omega,\R)}\\
    & \qquad + L^2 \|\big(1+|X(t_1)|\big)^q \|_{L^2(\Omega,\R)} |\sigma -s| \\
    & \quad \le C \Big( \big(1+ \sup_{t\in [0,T]}\|X(t)
    \|_{L^{4q-2}(\Omega,\R^d)}^{2q-1}\big) |\sigma - t_1|^{\frac{1}{2}} \\
    & \qquad \qquad + \big(1+\sup_{t\in [0,T]}\|X(t)
    \|_{L^{2q}(\Omega,\R^d)}^{q}\big)|\sigma -s| \Big). 
  \end{align*}
   To sum up, we have shown
  \begin{align*}
    \Gamma_2(\tau) \le C \big( 1 + \sup_{t \in [0,T]} \| X(t) \|_{L^{4q -
    2}(\Omega;\R^d)}^{2q-1} \big)  |t_1 - t_2|. 
  \end{align*}
 Since $4q-2 \le 6q-4$, this completes the proof.
\end{proof}
The proof shows that it is sufficient to have bounds for moments of order 
$\max(5q-3,4q)$ instead of $6q-4$.
However, in view of the weaker estimate in Lemma \ref{lem:cons1b},
this does not improve the result of Theorem \ref{th:PMilcons}.

Now we are well-prepared for the proof of Theorem~\ref{th:PMilcons}.

\begin{proof}[Proof of Theorem~\ref{th:PMilcons}]
  We first verify \eqref{eq:cons_cond1} with $\gamma = 1$ for the PMil method.
  For this, let $(t,\delta) \in \mathbb{T}$ be arbitrary. After inserting
  \eqref{exact} and \eqref{eq:PsiPMil} we obtain
  \begin{align*}
    &\big\| \E \big[ X(t + \delta) - \Psi^{\mathrm{PMil}}(X(t),t,\delta) \big|
    \F_{t} \big] \big\|_{L^2(\Omega;\R^d)}\\
    &\quad = \Big\| \E \Big[ X(t) + \int_{t}^{t+\delta} f(\tau,X(\tau))
    \diff{\tau} - X^{\circ}(t) - \delta f(t,X^{\circ}(t)) \Big| \F_{t} \Big]
    \Big\|_{L^2(\Omega;\R^d)}\\
    &\quad \le \big\| X(t) - X^{\circ}(t) \big\|_{L^2(\Omega;\R^d)} + \delta
    \big\| f(t,X(t)) - f(t,X^{\circ}(t)) \big\|_{L^2(\Omega;\R^d)}\\
    &\qquad + \int_{t}^{t+\delta} \big\| \E \big[ f(\tau,X(\tau)) - f(t,X(t))
    \big| \F_t \big] \big\|_{L^2(\Omega;\R^d)} \diff{\tau}.
  \end{align*}
  By applying Lemma~\ref{lem:PMilcons1} with $\varphi = \id$, $\kappa =
  1$, $p = 8q - 6$, and $\alpha = \frac{1}{2(q-1)}$ we obtain 
  \begin{align*}
    \big\| X(t) - X^{\circ}(t) \big\|_{L^2(\Omega;\R^d)} \le 
    C \big( 1 + \big\| X(t) \big\|_{L^{8q - 6}(\Omega;\R^d)}^{4q - 3}
    \big) \delta^2, 
  \end{align*}
  since $\frac{1}{2}\alpha(p-2)\kappa = 2$.
  Similarly, we estimate the second term by Lemma~\ref{lem:PMilcons1} 
  with $\varphi = f(t,\cdot)$, $\kappa = q$, and $p = 6 - \frac{4}{q}$. Since
  in this case $\frac{1}{2}\alpha(p-2)\kappa = 1$ we get
  \begin{align}
    \label{eq:term3}
    \delta \big\| f(t,X(t)) - f(t,X^{\circ}(t)) \big\|_{L^2(\Omega;\R^d)}
    \le C \big( 1 + \big\| X(t) \big\|_{L^{6q - 4}(\Omega;\R^d)}^{3q-2}
    \big) \delta^2.
  \end{align}
  The last term is estimated by Lemma~\ref{lem:cons1b} with $t_1 = s = t$ and
  $t_2= t+\delta$,
  \begin{align*}
    &\int_{t}^{t+\delta} \big\| \E \big[ f(\tau,X(\tau)) - f(t,X(t))
    \big| \F_t \big] \big\|_{L^2(\Omega;\R^d)} \diff{\tau} \\
    &\quad \le C \big( 1 + \sup_{t\in[0,T]} \big\| X(t) \big\|_{L^{6q -
    4}(\Omega;\R^d)}^{3q-2} \big) \delta^2.
  \end{align*}
  This completes the proof of \eqref{eq:cons_cond1}. For the proof of
  \eqref{eq:cons_cond2} we first insert \eqref{exact} and \eqref{eq:PsiPMil}.
  Then, in the same way as above we obtain the following four terms
  \begin{align*}
    &\big\| ( \id - \E [\,\cdot\,|\F_t]) \big( X(t + \delta) -
    \Psi^{\mathrm{PMil}}(X(t),t,\delta) \big) \big\|_{L^2(\Omega;\R^d)}\\  
    &\le \sum_{r = 1}^m \big\| \big( g^r(t,X(t)) - g^r(t,X^{\circ}(t))
    \big) I_{(r)}^{t,t+\delta} \big\|_{L^2(\Omega;\R^d)} \\
    &+\sum_{r,r_2 = 1}^m \big\| \big( g^{r,r_2}(t,X(t)) -
    g^{r,r_2}(t,X^{\circ}(t)) \big) I_{(r,r_2)}^{t,t+\delta}
    \big\|_{L^2(\Omega;\R^d)}\\ 
    &+\Big\| ( \id - \E [\,\cdot\,|\F_t]) \int_{t}^{t+\delta}
    f(\tau,X(\tau)) \diff{\tau} \Big\|_{L^2(\Omega;\R^d)}\\
    &+ \sum_{r = 1}^m \Big\| \int_{t}^{t + \delta} g^{r}(\tau,X(\tau)) -
    g^{r}(t,X(t)) \diff{W^{r}(\tau)} - \sum_{r_2 = 1}^m 
    g^{r,r_2}(t,X(t)) I_{(r_2,r)}^{t,t+\delta}  \Big\|_{L^2(\Omega;\R^d)}.
  \end{align*}
  Since the stochastic increment $I_{(r)}^{t,t+\delta}$ is independent of
  $\F_t$ it directly follows that
  \begin{align*}
    &\big\| \big( g^r(t,X(t)) - g^r(t,X^{\circ}(t))
    \big) I_{(r)}^{t,t+\delta} \big\|_{L^2(\Omega;\R^d)}\\
    &\quad = \delta^{\frac{1}{2}} \big\| 
    g^r(t,X(t)) - g^r(t,X^{\circ}(t)) \big\|_{L^2(\Omega;\R^d)}.
  \end{align*}
  Then, we apply Lemma~\ref{lem:PMilcons1} 
  with $\varphi = g^r(t,\cdot)$, $\kappa = \frac{q+1}{2}$, and $p = 10 -
  \frac{16}{q+1}$. As above, this yields $\frac{1}{2}\alpha(p-2)\kappa = 1$ and   we get 
  \begin{align*}
    \delta^{\frac{1}{2}} \big\| g^r(t,X(t)) - g^r(t,X^{\circ}(t))
    \big\|_{L^2(\Omega;\R^d)} \le C \big( 1 + \sup_{t\in[0,T]} \big\| X(t)
    \big\|_{L^{5q - 3}(\Omega;\R^d)}^{\frac{5}{2}q- \frac{3}{2}} \big)
    \delta^{\frac{3}{2}}
  \end{align*}
  for every $r = 1,\ldots,m$. In the same way we obtain for the second term 
  \begin{align*}
    &\big\| \big( g^{r,r_2}(t,X(t)) - g^{r,r_2}(t,X^{\circ}(t))
    \big) I_{(r,r_2)}^{t,t+\delta} \big\|_{L^2(\Omega;\R^d)}\\
    &\quad = \frac{1}{\sqrt{2}} \delta \big\| 
    g^{r,r_2}(t,X(t)) - g^{r,r_2}(t,X^{\circ}(t)) \big\|_{L^2(\Omega;\R^d)}.
  \end{align*}
  Then, a further application of Lemma~\ref{lem:PMilcons1} 
  with $\varphi = g^{r,r_2}(t,\cdot)$, $\kappa = q$, and $p = 4 -
  \frac{2}{q}$ gives
  \begin{align*}
    \delta \big\| g^{r,r_2}(t,X(t)) - g^{r,r_2}(t,X^{\circ}(t))
    \big\|_{L^2(\Omega;\R^d)} \le C \big( 1 + \sup_{t\in[0,T]} \big\| X(t)
    \big\|_{L^{4q - 2}(\Omega;\R^d)}^{2q- 1} \big)
    \delta^{\frac{3}{2}}    
  \end{align*}
  for every $r,r_2 = 1,\ldots,m$, since in this case
  $\frac{1}{2}\alpha(p-2)\kappa = \frac{1}{2}$. 
    
  Next, since $f(t,X(t))$ is $\F_t$-measurable it follows for the third term
  that
  \begin{align*}
    &\Big\| ( \id - \E [\,\cdot\,|\F_t]) \int_{t}^{t+\delta}
    f(\tau,X(\tau)) \diff{\tau} \Big\|_{L^2(\Omega;\R^d)}\\
    &\quad = \Big\| ( \id - \E [\,\cdot\,|\F_t]) \int_{t}^{t+\delta}
    f(\tau,X(\tau)) - f(t,X(t)) \diff{\tau} \Big\|_{L^2(\Omega;\R^d)}.    
  \end{align*}
  By making use of $\| ( \id - \E [\,\cdot\,|\F_t]) Y
  \|_{L^2(\Omega;\R^d)} \le \| Y \|_{L^2(\Omega;\R^d)}$ one directly deduces
  the desired estimate from Lemma~\ref{lem:cons1}. Finally, the last
  term is estimated by Lemma~\ref{lem:cons2}.
\end{proof}

\section{C-stability of the split-step backward Milstein method}
\label{sec:SSBMstab}

In this section we verify that Assumption~\ref{as:fg} and condition
\eqref{eq2:onesidedSSBM} are sufficient for the C-stability of the split-step
backward Milstein method.

The results of Proposition~\ref{prop:homeomorph} below are needed in order to
show that the SSBM method is a well-defined one-step method in the sense of
Definition~\ref{def:onestep}. Further, the inequality \eqref{eq:stab} plays a
key role in the proof of the C-stability of the SSBM method and generalizes a
similar estimate for the split-step backward Euler method from
\cite[Corollary~4.2]{beyn2015}. 

\begin{prop}
  \label{prop:homeomorph}
  Let the functions $f \colon [0,T] \times \R^d \to \R^d$ and $g^r \colon [0,T]
  \times \R^d \to \R^d$, $r = 1,\ldots,m$, satisfy Assumption~\ref{as:fg} and
  condition \eqref{eq2:onesidedSSBM} with
  $L \in (0,\infty)$, $\eta_1 \in (1,\infty)$, and
  $\eta_2 \in (0,\infty)$. Let
  $\overline{h} \in (0,L^{-1})$ be given and define for every $\delta \in
  (0,\overline{h}]$ the mapping $F_\delta \colon [0,T] \times \R^d \to \R^d$ by
  $F_\delta(t,x) = x - \delta f(t, x)$. Then, the mapping $\R^d \ni x \mapsto
  F_\delta(t,x) \in \R^d$ is a homeomorphism for every $t \in [0,T]$. 
  
  In addition, the inverse $F_\delta^{-1}(t,\cdot) \colon \R^d
  \to \R^d$ satisfies
  \begin{align}
    \label{eq:Fhinv_lip}
    \big| F_\delta^{-1}(t,x_1)-F_\delta^{-1}(t,x_2) \big| &\le (1 - L
    \delta)^{-1} | x_1 - x_2 |, \\ 
    \label{eq:Fhinv_growth}
    \big| F_\delta^{-1}(t,x) \big| &\le(1 - L
    \delta)^{-1} \big( L \delta + | x | \big),
  \end{align}
  for every $x,x_1, x_2 \in \R^d$ and $t \in [0,T]$. Moreover, there
  exists a constant $C_1$ only depending on $L$ and $\overline{h}$ such
  that
  \begin{align}
    \label{eq:stab}
    \begin{split}
      &\big| F_\delta^{-1}(t,x_1) - F_\delta^{-1}(t,x_2) \big|^2 + \eta_1
      \delta \sum_{r = 1}^m \big| g^r(t, F_\delta^{-1}(t,x_1)) - g^r(t,
      F_\delta^{-1}(t,x_2)) \big|^2\\
      &\quad + \eta_2 \delta \sum_{r_1,r_2 = 1}^{m} \big| 
      g^{r_1,r_2}(t, F_\delta^{-1}(t,x_1)) - 
      g^{r_1,r_2}(t,F_\delta^{-1}(t,x_2)) \big|^2 \le (1 + C_1
      \delta) \big| x_1 - x_2 \big|^2 
    \end{split}
  \end{align}
  for every $x_1, x_2 \in \R^d$ and $t \in [0,T]$.
\end{prop}

\begin{proof}
  The first part is a direct consequence of the
  Uniform Monotonicity Theorem (see for instance, \cite[Chap.6.4]{ortega2000},  
  \cite[Theorem~C.2]{stuart1996}). The estimates \eqref{eq:Fhinv_lip} and
  \eqref{eq:Fhinv_growth} are standard and a proof is found, for example, in
  \cite[Sec.~4]{beyn2015}. 

  Regarding \eqref{eq:stab} it first follows from \eqref{eq2:onesidedSSBM} that
  \begin{align*}
    &\langle F_\delta(t,x_1) - F_\delta(t,x_2), x_1 - x_2 \rangle\\
    &\quad = | x_1 - x_2 |^2 - \delta 
    \langle f(t,x_1) - f(t,x_2), x_1 - x_2 \rangle\\
    &\quad\ge (1 - L \delta) | x_1 - x_2 |^2 + \eta_1 \delta \sum_{r = 1}^m
    \big| g^r(t,x_1) - g^r(t,x_2) \big|^2\\
    &\qquad + \eta_2 \delta \sum_{r_1,r_2 = 1}^m
    \big| g^{r_1,r_2}(t,x_1) - g^{r_1,r_2}(t,x_2) \big|^2
  \end{align*}
  for all $x_1, x_2 \in \R^d$. For some $y_1, y_2 \in \R^d$ we substitute $x_1
  = F_\delta^{-1}(t,y_1)$ and $x_2 = F_\delta^{-1}(t,y_2)$ into this
  inequality. Then, after some rearranging we obtain
  \begin{align*}
    &\big|F_\delta^{-1}(t,y_1) - F_\delta^{-1}(t,y_2)\big|^2 + \eta_1 \delta
    \sum_{r = 1}^m \big| g^r(t,F_\delta^{-1}(t,y_1)) -
    g^r(t,F_\delta^{-1}(t,y_2)) \big|^2\\
    &\qquad + \eta_2 \delta \sum_{r_1,r_2 = 1}^m
    \big| g^{r_1,r_2}(t,F_\delta^{-1}(t,y_1)) -
    g^{r_1,r_2}(t,F_\delta^{-1}(t,y_2)) \big|^2 \\
    &\quad \le \big\langle y_1 - y_2, F_\delta^{-1}(t,y_1) -
    F_\delta^{-1}(t,y_2) \big\rangle + L \delta \big| F_\delta^{-1}(t,y_1) -
    F_\delta^{-1}(t,y_2)\big|^2. 
  \end{align*}
  Next, as in the proof of \cite[Corollary~4.2]{beyn2015} we apply 
  the Cauchy-Schwarz inequality and \eqref{eq:Fhinv_lip}. This yields
  \begin{align*}
    &\big\langle y_1 - y_2, F_\delta^{-1}(t,y_1) -
    F_\delta^{-1}(t,y_2) \big\rangle + L \delta \big| F_\delta^{-1}(t,y_1) -
    F_\delta^{-1}(t,y_2)\big|^2\\ 
    &\quad \le |y_1 - y_2 |  \big|F_\delta^{-1}(t,y_1) -
    F_\delta^{-1}(t,y_2)\big| +  L \delta \big| F_\delta^{-1}(t,y_1) -
    F_\delta^{-1}(t,y_2)\big|^2\\
    &\quad \le(1 - L \delta)^{-1} \big( 1 + (1 - L \delta)^{-1} L \delta \big)
    |y_1 - y_2 |^2 =(1-L \delta)^{-2}|y_1-y_2|^2
  \end{align*}
  for all $y_1, y_2 \in \R^d$.
  % \begin{align*}
  %   (1 - L \delta)^{-1} = 1 + (1 - L \delta)^{-1} \big( 1  - (1 - L \delta)
  %   \big) = 1 + (1 - L \delta)^{-1} L \delta
  % \end{align*}
  % as well as
  Finally, note that $b(\delta)=(1-L \delta)^{-2}$ is a convex function, hence
  for all $\delta \in [0,\overline{h}]$, 
  \begin{align*}
    (1 - L \delta)^{-2} \le 1 + C_1 \delta , \quad \text{with }
    C_1 = \frac{b(\overline{h})-b(0)} {\overline{h}}= L(2-L\overline{h}) 
    (1- L \overline{h})^{-2},
  \end{align*}
  and inequality \eqref{eq:stab} is verified.
\end{proof}

Proposition~\ref{prop:homeomorph} ensures that the implicit step of
the SSBM method admits a unique solution if $f$ satisfies
Assumption~\ref{as:fg} with one-sided Lipschitz constant $L$. To be more 
precise, for a given $\overline{h} \in (0,L^{-1})$ let us consider an arbitrary
vector of step sizes $h \in (0,\overline{h}]^N$, $N \in \N$. Then, it follows 
from Proposition~\ref{prop:homeomorph} that the nonlinear equations 
\begin{align*}
  \overline{X}_h^{\mathrm{SSBM}}(t_i) &= X_h^{\mathrm{SSBM}}(t_{i-1}) + h_i
  f(t_{i}, \overline{X}_h^{\mathrm{SSBM}}(t_i)), \quad 1 \le i \le N, 
\end{align*}
are uniquely solvable. Further, there exists a homeomorphism
$F_{h_i}(t_i,\cdot) \colon \R^d \to \R^d$ such that
$\overline{X}_h^{\mathrm{SSBM}}(t_i) =
F_{h_i}^{-1}(t_i,X_h^{\mathrm{SSBM}}(t_{i-1}))$. Therefore, the
one-step map $\Psi^{\mathrm{SSBM}} \colon \R^d \times \mathbb{T} \times \Omega
\to \R^d$ of the split-step backward Milstein method is given by
\begin{align}
  \label{eq:PsiSSBM}
  \begin{split}
    \Psi^{\mathrm{SSBM}}(x,t,\delta) &= F_{\delta}^{-1}(t+\delta,x) +
    \sum_{r=1}^m g^r(t + \delta, F_\delta^{-1}(t+\delta, x))
    I_{(r)}^{t,t+\delta}\\
    &\quad + \sum_{r_1, r_2 = 1}^m g^{r_1,r_2}(t+\delta,
    F_{\delta}^{-1}(t+\delta,x) ) I_{(r_2,r_1)}^{t,t + \delta}
  \end{split}
\end{align}
for every $x \in \R^d$ and $(t,\delta) \in \mathbb{T}$, where the stochastic
increments are defined in \eqref{eq2:stochincr1} and \eqref{eq2:stochincr2}.
Next, we verify that $\Psi^{\mathrm{SSBM}}$ satisfies condition
\eqref{eq:Psicond} as well as \eqref{eq3:cond1} and \eqref{eq3:cond2}.

\begin{prop}  
  \label{prop:SSBM}
  Let the functions $f$ and $g^r$, $r = 1,\ldots,m$, satisfy
  Assumption~\ref{as:fg} and condition \eqref{eq2:onesidedSSBM} with $L \in
  (0,\infty)$, $q \in [2,\infty)$, $\eta_1 \in (1,\infty)$, and $\eta_2 \in
  (0,\infty)$. For every $\overline{h} \in (0, \max(L^{-1},\frac{2
  \eta_2}{\eta_1}))$ and initial value $\xi \in L^2(\Omega;\F_{0},\P;\R^d)$ 
  it holds that $(\Psi^{\mathrm{SSBM}}, \overline{h}, \xi)$ is a stochastic
  one-step method.  
  
  In addition, there exists a constant $C_0$ depending on $L$, $q$, $m$,
  and $\overline{h}$, such that 
  \begin{align}
    \label{eq:SSBMzero1}
    \big\| \E \big[ \Psi^{\mathrm{SSBM}}( 0, t,\delta) | \F_{t} \big]
    \big\|_{L^2(\Omega;\R^d)} &\le C_0 \delta,\\
    \label{eq:SSBMzero2}
    \big\| \big( \id - \E [ \, \cdot\, | \F_{t} ] \big)
    \Psi^{\mathrm{SSBM}}( 0, t,\delta) \big\|_{L^2(\Omega;\R^d)} & \le C_0
    \delta^{\frac{1}{2}}    
  \end{align}
  for all $(t,\delta) \in \mathbb{T}$.
\end{prop}

\begin{proof}
  Regarding the first assertion we show that $\Psi^{\mathrm{SSBM}}$ 
  satisfies \eqref{eq:Psicond}. For this we fix arbitrary $(t, \delta) \in
  \mathbb{T}$ and $Z \in L^2(\Omega,\F_t,\P;\R^d)$. Then, we obtain from
  Proposition~\ref{prop:homeomorph} that the mapping $F_\delta^{-1}(t+\delta,
  \cdot) \colon \R^d \to \R^d$ is a homeomorphism satisfying the linear growth
  bound \eqref{eq:Fhinv_growth}. Hence, we have 
  \begin{align*}
    F_\delta^{-1}(t+\delta, Z) \in L^2(\Omega,\F_t,\P;\R^d).
  \end{align*}
  Consequently, by the continuity of $g^r$ and $g^{r_1,r_2}$ the mappings
  \begin{align*}
    \Omega \ni \omega \mapsto g^r(t+\delta, F_\delta^{-1}(t+\delta, Z(\omega)))
    \in \R^d
  \end{align*}
  and 
  \begin{align*}
    \Omega \ni \omega \mapsto g^{r_1,r_2}(t+\delta, F_\delta^{-1}(t+\delta,
    Z(\omega))) \in \R^d
  \end{align*}
  are $\F_t / \B(\R^d)$-measurable for every $r,r_1,r_2 = 1,\ldots,m$. Hence,
  $\Psi^{\mathrm{SSBM}}(Z,t,\delta) \colon \Omega \to \R^d$ is an
  $\F_{t+\delta} / \B(\R^d)$-measurable random variable. 
  
  Next, we show that $\Psi^{\mathrm{SSBM}}(Z,t,\delta)$ is square integrable.
  First, it follows from \eqref{eq:Fhinv_growth} that 
  \begin{align*}
    \big\| \E \big[ \Psi^{\mathrm{SSBM}}( Z, t,\delta) | \F_{t} \big]
    \big\|_{L^2(\Omega;\R^d)} &=
    \big\| F_\delta^{-1}(t+\delta, Z) \big\|_{L^2(\Omega;\R^d)}\\
    &\le (1 - L\delta)^{-1} \big( L \delta + \|Z\|_{L^2(\Omega;\R^d)} \big). 
  \end{align*}
  In particular, if $Z = 0 \in L^2(\Omega;\R^d)$ we get 
  \begin{align*}
    \big\| \E \big[ \Psi^{\mathrm{SSBM}}( 0, t,\delta) | \F_{t} \big]
    \big\|_{L^2(\Omega;\R^d)} \le (1 - L\overline{h})^{-1} L \delta,
  \end{align*}
  which is \eqref{eq:SSBMzero1}. Further, since the stochastic increments
  $I_{(r)}^{t,t+\delta}$ and $I_{(r_1,r_2)}^{t,t+\delta}$ are pairwise
  uncorrelated and satisfy $\E[ |I_{(r)}^{t,t+\delta}|^2 ] = \delta$ and
  $\E[ |I_{(r_1,r_2)}^{t,t+\delta}|^2 ] = \frac{1}{2} \delta^2$ we obtain 
  \begin{align*}
    &\big\| \big( \id - \E [ \, \cdot\, | \F_{t} ] \big)
    \Psi^{\mathrm{SSBM}}( 0, t,\delta) \big\|_{L^2(\Omega;\R^d)}^2\\
    &\quad = \Big\| \sum_{r = 1}^m g^r(t+\delta, F_\delta^{-1}(t+\delta, 0) )
    I_{(r)}^{t,t+\delta}\Big\|^2_{L^2(\Omega;\R^d)}\\
    &\qquad + \Big\| \sum_{r_1, r_2 = 1}^m g^{r_1,r_2}(t+\delta, 
    F_{\delta}^{-1}(t+\delta,0) ) I_{(r_2,r_1)}^{t,t + \delta}
    \Big\|^2_{L^2(\Omega;\R^d)}\\ 
    &\quad = \delta \sum_{r = 1}^m \big| g^r(t+\delta, F_\delta^{-1}(t+\delta,
    0) ) \big|^2 + \frac{1}{2} \delta^2 \sum_{r_1,r_2 =
    1}^m \big| g^{r_1,r_2} (t+\delta, F_\delta^{-1}(t+\delta, 0) ) \big|^2.
  \end{align*}
  Then, applications of \eqref{eq3:poly_growth_g}
  and \eqref{eq:Fhinv_growth} yield
  \begin{align*}
    \big| g^r(t+\delta, F_\delta^{-1}(t+\delta,
    0) ) \big| &\le L \big( 1 + \big| F_\delta^{-1}(t+\delta,
    0) \big| \big)^{\frac{q+1}{2}}\\
    &\le L \big( 1 + (1 - L
    \overline{h})^{-1} L \overline{h} \big)^{\frac{q+1}{2}}  
  \end{align*}
  and, similarly, by \eqref{eq3:grr_poly_growth}
  \begin{align*}
    \big| g^{r_1,r_2}(t+\delta, F_\delta^{-1}(t+\delta,
    0) ) \big| &\le L \big( 1 + \big| F_\delta^{-1}(t+\delta,
    0) \big| \big)^q\\
    &\le L \big( 1 + ( 1 - L \overline{h})^{-1}
    L \overline{h} \big)^{q}. 
  \end{align*}
  Therefore, there exists a constant $C_0$ depending on $L$, $q$, $m$, and
  $\overline{h}$, such that \eqref{eq:SSBMzero2} is satisfied. In particular,
  this proves that $\Psi^{\mathrm{SSBM}}( 0,t,\delta) \in
  L^2(\Omega,\F_{t+\delta},\P;\R^d)$.

  Next, for arbitrary $Z \in L^2(\Omega;\F_t,\P;\R^d)$ the same arguments as
  above yield
  \begin{align*}
    &\big\| \Psi^{\mathrm{SSBM}}(Z,t,\delta) -  \Psi^{\mathrm{SSBM}}( 0,
    t,\delta) \big\|^2_{L^2(\Omega;\R^d)}\\
    &\quad = \big\|  F_\delta^{-1}(t+\delta, Z) - F_\delta^{-1}(t+\delta, 0)
    \big\|_{L^2(\Omega;\R^d)}^2 \\
    &\qquad + \delta \sum_{r = 1}^m \big\|
    g^r(t+\delta, F_\delta^{-1}(t+\delta, Z) ) -
    g^r(t+\delta, F_\delta^{-1}(t+\delta, 0) )
    \big\|_{L^2(\Omega;\R^d)}^2\\
    &\qquad + \frac{\delta^2}{2} \sum_{r_1,r_2 = 1}^m \big\|
    g^{r_1,r_2}(t+\delta, F_\delta^{-1}(t+\delta, Z) ) -
    g^{r_1,r_2}(t+\delta, F_\delta^{-1}(t+\delta, 0) )
    \big\|_{L^2(\Omega;\R^d)}^2.
  \end{align*}
  Note that $\eta_1 > 1$ and $\frac{\delta^2}{2} \le \frac{\overline{h}}{2}
  \delta \le \eta_2 \delta$. Thus, the inequality \eqref{eq:stab} is applicable
  and we obtain 
  \begin{align*}
    &\big\| \Psi^{\mathrm{SSBM}}(Z,t,\delta) -  \Psi^{\mathrm{SSBM}}( 0,
    t,\delta) \big\|^2_{L^2(\Omega;\R^d)} \le (1 + C_1 \delta) \| Z
    \|^2_{L^2(\Omega;\R^d)}.    
  \end{align*}
  Hence $\Psi^{\mathrm{SSBM}}(Z,t,\delta) \in
  L^2(\Omega,\F_{t+\delta},\P;\R^d)$. 
\end{proof}

\begin{theorem}
  \label{th:SSBMstab}
  Let the functions $f$ and $g^r$, $r = 1,\ldots,m$, satisfy
  Assumption~\ref{as:fg} and condition \eqref{eq2:onesidedSSBM} with $L \in
  (0,\infty)$, $\eta_1 \in (1,\infty)$, and $\eta_2 \in (0,\infty)$. Further,
  let $\overline{h} \in (0, \max(L^{-1},\frac{2 \eta_2}{\eta_1}))$. Then, for
  every $\xi \in L^2(\Omega,\F_0,\P;\R^d)$ the SSBM scheme
  $(\Psi^{\mathrm{SSBM}},\overline{h},\xi)$ is stochastically C-stable.
\end{theorem}

\begin{proof}
  Let $(t,\delta) \in \mathbb{T}$ be arbitrary. For every $Y \in
  L^2(\Omega,\F_t,\P;\R^d)$ we have
  \begin{align*}
    \E \big[ \Psi^{\mathrm{SSBM}}(Y,t,\delta)| \F_t \big] = 
    F_{\delta}^{-1}(t+\delta,Y)
  \end{align*}
  and
  \begin{align*}
    \big( \id - \E[\, \cdot \, | \F_t] \big)
    \Psi^{\mathrm{SSBM}}(Y,t,\delta) 
    &= \sum_{r = 1}^m g^r(t+\delta, F_\delta^{-1}(t+\delta,Y))
    I_{(r)}^{t,t+\delta} \\
    &\quad + \sum_{r_1, r_2 = 1}^m g^{r_1,r_2}(t + 
    \delta, F_{\delta}^{-1}(t+\delta,Y) ) I_{(r_2,r_1)}^{t,t + \delta}.
  \end{align*}
  For the computation of the $L^2$-norm, we make use of the facts that the
  stochastic increments are independent of $\F_t$ and pairwise 
  uncorrelated. Further, since $\E[ |I_{(r)}^{t,t+\delta}|^2 ] = \delta$ and
  $\E[ |I_{(r_1,r_2)}^{t,t+\delta}|^2 ] = \frac{1}{2} \delta^2$ it follows for
  \eqref{eq:stab_cond1} with $\nu = \eta_1$
  \begin{align*}
    &\big\| \E \big[ \Psi^{\mathrm{SSBM}}(Y,t,\delta) -
    \Psi^{\mathrm{SSBM}}(Z,t,\delta)| \F_t \big] 
    \big\|^2_{L^2(\Omega;\R^d)}\\
    &\qquad + \nu \big\| \big( \id - \E[ \cdot |
    \F_t] \big) \big( \Psi^{\mathrm{SSBM}}(Y,t,\delta) -
    \Psi^{\mathrm{SSBM}}(Z,t,\delta) \big) \big\|^2_{L^2(\Omega;\R^d)}\\
    &\quad = \E \big[ \big| F_{\delta}^{-1}(t+\delta,Y) -
    F_{\delta}^{-1}(t+\delta,Z) \big|^2 \big] \\
    &\qquad + \eta_1 \delta \sum_{r = 1}^m \E \big[ \big| g^r(t+\delta,
    F_\delta^{-1}(t+\delta,Y)) - g^r(t+\delta,F_\delta^{-1}(t+\delta,Z))
    \big|^2 \big]\\
    &\qquad + \frac{1}{2} \eta_1 \delta^2 \sum_{r_1,r_2 = 1}^{m} \E \big[ \big| 
    g^{r_1,r_2}(t+\delta, F_\delta^{-1}(t+\delta,Y)) -
    g^{r_1,r_2}(t+\delta,F_\delta^{-1}(t+\delta,Z)) \big|^2 \big].
  \end{align*}
  Due to $\frac{1}{2} \eta_1 \delta \le \frac{1}{2} \eta_1 \overline{h} <
  \eta_2$ an application of inequality \eqref{eq:stab} yields
  \begin{align*}
    &\big\| \E \big[ \Psi^{\mathrm{SSBM}}(Y,t,\delta) -
    \Psi^{\mathrm{SSBM}}(Z,t,\delta)| \F_t \big] 
    \big\|^2_{L^2(\Omega;\R^d)}\\
    &\qquad + \nu \big\| \big( \id - \E[\, \cdot \, |
    \F_t] \big) \big( \Psi^{\mathrm{SSBM}}(Y,t,\delta) -
    \Psi^{\mathrm{SSBM}}(Z,t,\delta) \big) \big\|^2_{L^2(\Omega;\R^d)}\\
    &\quad \le (1 + C_1 \delta) \big\| Y - Z \big\|_{L^2(\Omega;\R^d)}^2,
  \end{align*}
  which is the C-stability condition \eqref{eq:stab_cond1} with $\nu = \eta_1
  \in (1,\infty)$.
\end{proof}

\section{B-consistency of the split-step backward Milstein method}
\label{sec:SSBMcons}

This section is devoted to the proof of the following result, which is
concerned with the B-consistency of the SSBM method.

\begin{theorem}
  \label{th:SSBMcons}
  Let the functions $f$ and $g^r$, $r = 1,\ldots,m$, satisfy
  Assumption~\ref{as:fg} with $L \in (0,\infty)$ and $q \in [2,\infty)$. Let
  $\overline{h} \in (0,L^{-1})$. If the exact solution $X$ to \eqref{sode}
  satisfies $\sup_{\tau \in [0,T]} \| X(\tau) \|_{L^{6q-4}(\Omega;\R^d)} <
  \infty$, then the split-step backward Milstein method
  $(\Psi^{\mathrm{SSBM}},\overline{h},X_0)$ is stochastically B-consistent of
  order $\gamma = 1$.  
\end{theorem}

For the proof we recall some estimates of the homeomorphism
$F_\delta^{-1}$ from \cite[Lemma~4.3]{beyn2015}, which will be useful for the
estimate of the local truncation error. 

\begin{lemma}
  \label{lem:est_homeo}
  Consider the same situation as in Proposition~\ref{prop:homeomorph}. 
  Then there exist constants $C_2$, $C_3$ only depending on $L$, 
  $\overline{h}$ and $q$ such that for every $\delta \in (0,\overline{h}]$ the
  inverse $F_\delta^{-1}(t,\cdot) \colon \R^d \to \R^d$ satisfies the 
  estimates 
  \begin{align}
    \label{eq:1ord}
    \big| F_\delta^{-1}(t,x) - x \big| &\le \delta C_2 \big( 1 + |
    x|^q \big),\\
    \label{eq:2ord}
    \big| F_\delta^{-1}(t,x) - x - \delta f(t,x) \big| &\le \delta^2 C_3 \big(
    1 + |x|^{2q-1} \big) 
  \end{align}
  for every $x \in \R^d$ and $t \in [0,T]$. 
\end{lemma}

\begin{proof}[Proof of Theorem~\ref{th:SSBMcons}]
  The proof follows the same steps as the proof of
  \cite[Theorem~5.7]{beyn2015}. Let us fix arbitrary $(t,\delta) \in
  \mathbb{T}$. Then, by inserting \eqref{exact} and \eqref{eq:PsiSSBM} we
  obtain the following representation of the local truncation error
  \begin{align*}
    &X(t + \delta) - \Psi^{\mathrm{SSBM}} (X(t),t,\delta)
    = X(t) + \delta f(t + \delta,X(t)) - F_{\delta}^{-1}(t+\delta,X(t))\\
    &\qquad + \int_{t}^{t + \delta} \big( f(\tau,X(\tau)) - f(t + \delta,X(t))
    \big) \diff{\tau}\\
    &\qquad + \Big( \sum_{r = 1}^m \int_{t}^{t + \delta} \big(
    g^r(\tau,X(\tau)) - g^r(t + \delta,X(t)) \big) \diff{W^r(\tau)}\\
    &\qquad\qquad - \sum_{r_1,r_2 = 1}^m g^{r_1,r_2}(t+\delta,X(t))
    I_{(r_2,r_1)}^{t,t+\delta} \Big) \\
    &\qquad +\sum_{r = 1}^m \big( g^r(t+ \delta,X(t)) -
    g^r(t+\delta,F_{\delta}^{-1}(t + \delta,X(t))) \big) I_{(r)}^{t,t+\delta}\\
    &\qquad +\sum_{r_1,r_2 = 1}^m \big(g^{r_1,r_2}(t+ \delta,X(t)) -
    g^{r_1,r_2}(t+\delta,F_{\delta}^{-1}(t + \delta,X(t))) \big)
    I_{(r_2,r_1)}^{t,t+\delta}\\
    &\quad =: T_1 + T_2 + T_3 + T_4 + T_5. 
  \end{align*}
  We discuss the five terms separately. It is already shown in the proof of
  \cite[Theorem~5.7]{beyn2015} that by applying \eqref{eq:2ord} the $L^2$-norm
  of the term $T_1$ is dominated by
  \begin{align}
    \label{eq7:T1}
    \begin{split}
      \| T_1 \|_{L^2(\Omega;\R^d)} &= \big\| X(t) + \delta f(t + \delta,X(t)) -
      F_{\delta}^{-1}(t + \delta,X(t)) \big\|_{L^2(\Omega;\R^d)}\\
      &\le C_3  \big\| 1 + | X(t) |^{2q-1} \big\|_{L^2(\Omega;\R)} \delta^{2}
      \\ 
      &\le C_3  \Big(1 + \sup_{\tau \in [0,T]} \| X(\tau)
      \|_{L^{4q - 2}(\Omega;\R^d)}^{2q-1} \Big) \delta^2.
    \end{split}
  \end{align}
  Moreover, if we consider the conditional expectation of the term
  $T_2$ with respect to $\F_t$, then after taking the $L^2$-norm we arrive at
  \begin{align*}
    \big\| \E \big[ T_2 | \F_t \big] \big\|_{L^2(\Omega;\R^d)} = 
    \Big\| \E \Big[ \int_{t}^{t + \delta} \big( f(\tau,X(\tau)) - f(t +
    \delta,X(t)) \big) \diff{\tau} \, \big| \F_t \Big]
    \Big\|_{L^2(\Omega;\R^d)}.
  \end{align*}
  Hence, an application of Lemma~\ref{lem:cons1b} with $t_1 =t$ and
  $s=t+\delta$ yields
  \begin{align}
    \label{eq7:T2}
    \big\| \E \big[ T_2 | \F_t \big] \big\|_{L^2(\Omega;\R^d)} 
    \le C \big( 1 + \sup_{t \in [0,T]} \big\| X(t) \big\|^{3q
    -2}_{L^{6q-4}(\Omega;\R^d)} \big) \delta^2.
  \end{align}
  Since $\E[ T_1 | \F_t ] = T_1$ and $\E[ T_i | \F_t ] = 0$ for $i \in
  \{3,4,5\}$ we get from \eqref{eq7:T1} and \eqref{eq7:T2}  
  \begin{align*}
    &\big\| \E \big[ X(t + \delta) - \Psi^{\mathrm{SSBM}} (X(t),t,\delta) \big|
    \F_t \big] \big\|_{L^2(\Omega;\R^d)}\\
    &\quad \le \| T_1 \|_{L^2(\Omega;\R^d)} 
    + \big\| \E \big[ T_2 | \F_t \big] \big\|_{L^2(\Omega;\R^d)}\\
    &\quad \le C \big( 1 + \sup_{t \in [0,T]} \big\| X(t) \big\|^{3q
    -2}_{L^{6q-4}(\Omega;\R^d)} \big) \delta^2.
  \end{align*}
  This proves \eqref{eq:cons_cond1} with $\gamma = 1$ and it remains to show
  \eqref{eq:cons_cond2}. For this we estimate
  \begin{align*}
    &\big\| \big( \id - \E[ \, \cdot\, | \F_t] \big)
    \big( X(t + \delta) - \Psi^{\mathrm{SSBM}} (X(t),t,\delta) \big)
    \big\|_{L^{2}(\Omega;\R^d)}\\
    &\quad \le \sum_{i = 1}^{5} \big\| \big( \id -
    \E[ \, \cdot\, | \F_t] \big) T_i \big\|_{L^{2}(\Omega;\R^d)}.
  \end{align*}
  Then, note that $\| ( \id - \E[ \, \cdot\, | \F_t] ) T_1
  \|_{L^{2}(\Omega;\R^d)} = 0$ since $T_1$ is $\F_t$-measurable. Further, by
  making use of the fact that $\| ( \id - \E[ \, \cdot\, | \F_t] ) Y
  \|_{L^{2}(\Omega;\R^d)} \le \| Y \|_{L^2(\Omega;\R^d)}$ for all $Y \in
  L^2(\Omega;\R^d)$ we get
  \begin{align*}
    \big\| \big( \id - \E[ \, \cdot\, | \F_t] \big) T_2
    \big\|_{L^{2}(\Omega;\R^d)} \le \| T_2 \|_{L^2(\Omega;\R^d)}.
  \end{align*}
  After inserting $T_2$ it follows from Lemma~\ref{lem:cons1} that
  \begin{align*}
    \| T_2 \|_{L^2(\Omega;\R^d)} &\le \int_{t}^{t+\delta} \big\| f(\tau,X(\tau))
    - f(t+\delta,X(t)) \big\|_{L^2(\Omega;\R^d)} \diff{\tau}\\
    &\le C \big( 1 + \sup_{t \in [0,T]} \big\| X(t)
    \big\|_{L^{4q-2}(\Omega;\R^d)}^{2q-1} \big) \delta^{\frac{3}{2}}.
  \end{align*}
  Regarding the term $T_3$ we first couple the summation indices $r = r_1$.
  Then, the triangle inequality yields
  \begin{align*}
    \| T_3 \|_{L^2(\Omega;\R^d)} &\le \sum_{r = 1}^m \Big\| \int_{t}^{t+\delta} 
    \big( g^{r}(\tau,X(\tau)) - g^r(t+\delta,X(t)) \big) \diff{W^r(\tau)}\\
    &\qquad - \sum_{r_2=1}^m g^{r,r_2}(t+\delta,X(t)) I_{(r_2,r)}^{t,t+\delta}
    \Big\|_{L^2(\Omega;\R^d)}.
  \end{align*}
  Hence, we are in the situation of Lemma~\ref{lem:cons2} with $t_1 = t$ and
  $t_2 = s = t + \delta$ and we obtain
  \begin{align*}
    \| T_3 \|_{L^2(\Omega;\R^d)} &\le C \big( 1 + \sup_{t \in [0,T]} \|
    X(t) \|^{3q-2}_{L^{6q-4}(\Omega;\R^d)} \big) \delta^{\frac{3}{2}}.
  \end{align*}
  The $L^2$-norm estimates of the remaining terms $T_4$ and $T_5$ follow the
  same line of arguments as the last part of the proof of
  \cite[Theorem~5.7]{beyn2015}. For instance, the term $T_5$ is estimated as
  follows: From \eqref{eq3:grr_loc_Lip}, \eqref{eq:Fhinv_growth}, and
  \eqref{eq:1ord} we obtain 
  \begin{align*}
    &\big| g^{r_1,r_2}(t + \delta,X(t)) - g^{r_1,r_2}(t +
    \delta,F_{\delta}^{-1}(t + \delta,X(t))) \big|\\
    &\quad \le L \big( 1 + | X(t) |  + |F_{\delta}^{-1}(t +
    \delta,X(t))| \big)^{q-1}
    \big|X(t) - F_{\delta}^{-1}(t + \delta,X(t)) \big| \\
    &\quad \le C_2 L \big( 1 + | X(t) |  + (1 - L \delta)^{-1} 
    ( L \delta + | X(t) | ) \big)^{q-1} \big( 1 + | X(t) |^{q} \big)
    \delta\\ 
    &\quad \le C \big( 1 + | X(t) |^{2q-1} \big)\delta, 
  \end{align*}
  for a constant $C$ only depending on $C_2$, $L$, $q$, and
  $\overline{h}$. Therefore,
  \begin{align*}
    &\Big\| \sum_{r_1,r_2 = 1}^m \big( g^{r_1,r_2}(t + \delta,X(t)) -
    g^{r_1,r_2}(t + \delta,F_{\delta}^{-1}(t + \delta,X(t))) \big)
    I_{(r_2,r_1)}^{t,t+\delta} \Big\|_{L^2(\Omega;\R^d)}^2\\
    &\quad = \frac{1}{2} \delta^2 \sum_{r_1,r_2 = 1}^m \big\| g^{r_1,r_2} (t +
    \delta,X(t)) - g^{r_1,r_2}(t + \delta,F_{\delta}^{-1}(t + \delta,X(t)))
    \big\|_{L^2(\Omega;\R^d)}^2 \\
    &\quad \le C m^2  \Big( 1 + \sup_{\tau \in [0,T]} \|
    X(\tau) \|^{2q-1}_{L^{4q-2}(\Omega;\R^d)} \Big) \delta^4, 
  \end{align*}
  which is the desired estimate of $T_5$. 
  The corresponding estimate of $T_4$ reads
  \begin{align*}
    \| T_4 \|_{L^2(\Omega;\R^d)} \le C m \Big( 1 + \sup_{\tau \in [0,T]} \|
    X(\tau) \|^{\frac{3}{2}q-\frac{1}{2}}_{L^{3q-1}(\Omega;\R^d)} \Big)
    \delta^{\frac{3}{2}}
  \end{align*}
  and is obtained in the same way as for $T_5$ but with \eqref{eq3:loc_Lip_g}
  in place of \eqref{eq3:grr_loc_Lip}.
  Altogether, this completes the proof of \eqref{eq:cons_cond2} with $\gamma =
  1$.
\end{proof}

\section{Numerical experiments}
\label{sec:exp}
In this section we perform several numerical experiments which 
illustrate the preceding theory for two characteristic examples.

\subsection*{Double well dynamics with multiplicative noise}
%\medskip
% \noindent
 
%{\bf Double well dynamics with multiplicative noise}\\
%\medskip 
%To certify the strong convergence of the suggested schemes
Consider the following stochastic differential equation
\begin{align}
\label{eq:DWD}
\begin{split}
 \diff X(t)&=X(t)(1-X(t)^2)\diff t+\sigma(1-X(t)^2)\diff W(t), \quad t \in
 [0,T],\\ 
 X(0)&=X_0,
 \end{split}
\end{align}
where $\sigma >0$. The coefficient functions of \eqref{eq:DWD} are given by 
$f(x):=x(1-x^2)$ and $g(x):=\sigma(1-x^2)$, $x\in\R$. Note that $-f$ is the
gradient of the double well potential $V(x) =  \frac{1}{4}
x^4 - \frac{1}{2} x^2$. 

In our experiments we compare
the projected Milstein scheme \eqref{eq:PMildef} and the split-step backward Milstein method
\eqref{eq:SSBMdef}. In additon, we compare with the projected Euler-Maruyama method (PEM) proposed in \cite{beyn2015},  
\begin{align}
\label{PEM}
\begin{split}
\overline X_h(t_i)&=\min(1,h_i^{-\alpha}|X_h(t_{i-1})|^{-1})X_h(t_{i-1}),\\
X_h(t_i)&=\overline X(t_i)+h_if(t_{i-1},\overline X_h(t_i))+\sum\limits_{r=1}
^mg^r(t_{i-1},\overline X_h(t_i))I_{(r)}^{t_{i-1},t_i}.
\end{split}
\end{align}
As before, we have $1\le i\le N$, $h\in(0,1]^N$, $X_h(0):=X_0$ ,
  and $\alpha=\frac{1}{2(q-1)}$.

The equation \eqref{eq:DWD} satisfies Assumption~\ref{as:fg} and 
the coercivity 
condition \eqref{eq:growthcond} with polynomial growth rate $q=3$. 
Table~\ref{tab:conv} contains the restrictions on $\eta$ and $\sigma$ 
which imply
condition \eqref{eq3:onesided} in Assumption~\ref{as:fg}.
Moreover, we summarize the $p$-th moment bounds such that the assumptions of Theorems~\ref{th:PMilconv}, 
\ref{cor:SSBMconv}, and Theorem 6.7 in \cite{beyn2015} are satisfied. 
\begin{table}[h]
  \caption{Convergence conditions for the three methods}
\label{tab:conv} 
\begin{tabular}{p{4cm}|p{2cm}p{2cm}p{2cm}} 
&   PEM  &   PMil   &   SSBM                 \\   
\hline
\vspace{0.05cm}
Moment bounds \\ for exact solution:   & $p=6q-4$  &  $p=8q-6$ 
& $p=6q-4$ \\
$\sup_{t\in[0,T]}\|X(t)\|_{L^p}<\infty$
\vspace{0.2cm}
\\   \hline
\vspace{0.05cm}
Global monotonicity\\  condition \eqref{eq3:onesided}:&  $ \sigma^2<1$  & 
$\sigma^2<1 $   & $ \sigma^2<1$ \\
$2\eta\sigma^2\le 1$ for some 
$\eta>\frac{1}{2}$
\vspace{0.2cm}
\\ \hline 
\\[-0.25cm]
Coercivity condition \eqref{eq:growthcond}: & $p=14$, & $p=18$, & $p=14$,  \\
$\frac{p-1}{2}\sigma^2\le1$
for  $p\in[2,\infty)$ &    $\sigma^2<\frac{2}{13} $ & $\sigma^2<\frac{2}{17}$  
&  $\sigma^2<\frac{2}{13}$
\vspace{0.3cm}
\\   \hline 
\end{tabular}
\end{table}

Note that the additional condition \eqref{eq2:onesidedSSBM} for the SSBM-method is never satisfied for our example, 
since the square of the term $g(x)g'(x)=-2\sigma^2x(1-x^2), 
x\in\R$ is of sixth order and cannot be controlled by the fourth
order $f$-term in condition \eqref{eq2:onesidedSSBM}.

Since there is no explicit expression for the solution of \eqref{eq:DWD} we 
replace the exact solution  by a numerical reference approximation obtained
with an extremely small step size $\Delta t=2^{-17}$.
For this step-size, the projection onto the $(\Delta t)^{-\alpha}$-ball
actually never occurs.
The implicit step in the
SSBM scheme employs Cardano's method in order to solve the 
nonlinear equation exactly. The parameter value was set to $\alpha=\frac{1}{4}$ 
as prescribed 
by the results in Section~\ref{sec:PMilstab} above and in \cite{beyn2015}.
Figure~\ref{fig1} shows the 
strong error of convergence for seven different step sizes 
$h=2^k\Delta t, k=7,\dots,13$. The parameter values are $\sigma=0.3$ and  
$X_0=2$, and the coercivity condition \eqref{eq:growthcond}  always
holds.
 
%\raphael{Needed: Since $q = 3$, what are the precise bounds on $p$ such that
%Theorems~\ref{th:PMilconv} and \ref{cor:SSBMconv} are satisfied? We should also
%comment on the additional condition \eqref{eq2:onesidedSSBM}, which is
%apparently always violated in this example.}

\begin{figure}[t]
\includegraphics[width=11cm]{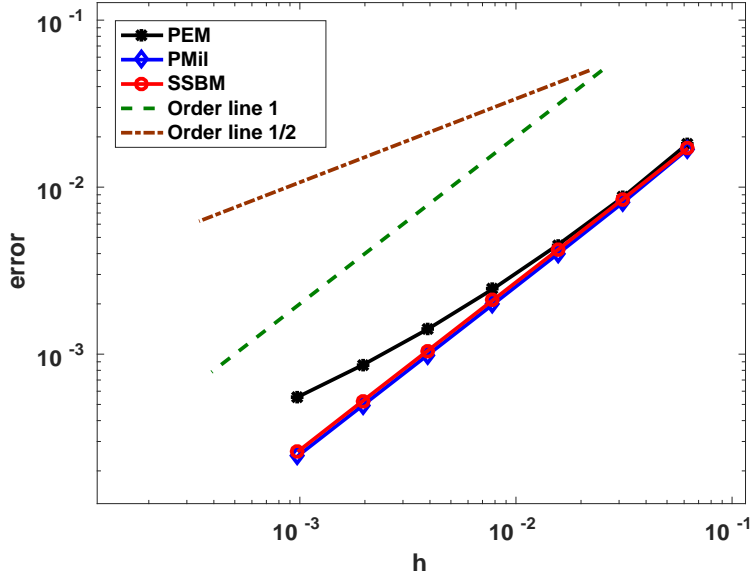}
\caption{Strong convergence errors for the approximation of the double well 
dynamics with parameter $\sigma=0.3$ 
 and $X_0=2$.  } 
\label{fig1}
\end{figure}
The strong error is measured at the endpoint $T=1$ by
\begin{align}
 \label{eq:error}
 \mathrm{error}=(\E[|X_h(T)-X(T)|^2])^\frac{1}{2},
\end{align}
with a  Monte Carlo simulation using $2\cdot10^6$ samples. For this number of 
samples we estimated the associated confidence intervals. They turned out to
be two orders of magnitude smaller than the values of the error itself for all 
methods and parameters shown. In the scale of Figure \ref{fig1} they will be 
hardly visible.

In Figure~\ref{fig1} one observes strong order $\gamma=1$ for the two 
Milstein-type schemes and  strong order $\gamma=\frac{1}{2}$ for the projected  Euler-Maruyama method, at least for smaller step-sizes. 

\begin{table}[h]
  \caption{Strong convergence errors for the approximation of the double well
  dynamics with parameter $\sigma=0.3$ and $X_0=2$.}
  \label{tab:ginz}
  \begin{tabular}{p{0.9cm}p{1.2cm}p{0.9cm}p{1cm}p{1.2cm}p{0.9cm}p{1cm}p{1.2cm}p{0.9cm}}
       &  PEM  &  & & PMil   & & & SSBM  &  
\\   \noalign{\smallskip}\hline\noalign{\smallskip}
  $h$   &   error  &  EOC &  \#  &   error   &  EOC  & \# & error  & EOC
\\   \noalign{\smallskip}\hline\noalign{\smallskip}
      $ 2^{-4}$ &       0.0183  &    &   101912 &       0.0169 &    &   152196 &        0.0171  \\
    $ 2^{-5}$ &       0.0087  &          1.07  &        0 &       0.0081 &          1.07  &        1 &       0.0085  &         1.01\\
    $ 2^{-6}$ &       0.0045  &          0.95  &        0 &       0.0040 &          1.02  &        0 &       0.0042  &         1.01\\
     $2^{-7}$ &       0.0025  &          0.88  &        0 &       0.0020 &          1.01  &        0 &       0.0021  &         1.00\\
     $2^{-8}$ &       0.0014  &          0.80  &        0 &       0.0010 &          1.00  &        0 &       0.0010  &         1.00\\
     $2^{-9}$ &       0.0009  &          0.71  &        0 &       0.0005 &          1.00  &        0 &       0.0005  &         1.00\\
     $2^{-10} $ &       0.0006  &          0.64  &        0 &       0.0002 &          1.01  &        0 &       0.0003  &         1.01
  \end{tabular}
\end{table}

Table~\ref{tab:ginz} contains the values of the computed errors and of the
corresponding experimental order of convergence defined by
\begin{align*}
\mathrm{EOC}=\frac{\log(\mathrm{error}(h_i))-\log(\mathrm{error}(h_{i-1}))}{
\log(h_i)-\log(h_{i-1})},
\end{align*}
which support the theoretical results.
Moreover, as in \cite{beyn2015} we are interested in the number of samples 
for which the trajectories of the PEM method and the PMil scheme leave the 
sphere of radius $h^{-\alpha}$, i.e. the total number of trajectories we
observed the following events
\begin{align}
\label{ineq:projec}
&\{i=1,\dots,N\colon|X^{\mathrm{PEM}}_h(t_i)|>h^{-\alpha}\} \neq \emptyset, \\
&\{i=1,\dots,N\colon|X^{\mathrm{PMil}}_h(t_i)|>h^{-\alpha}\} \neq \emptyset,
\end{align}
respectively. This information is provided in the fourth and 
seventh columns of Table~\ref{tab:ginz}.

\begin{figure}[t]
\includegraphics[width=11cm]{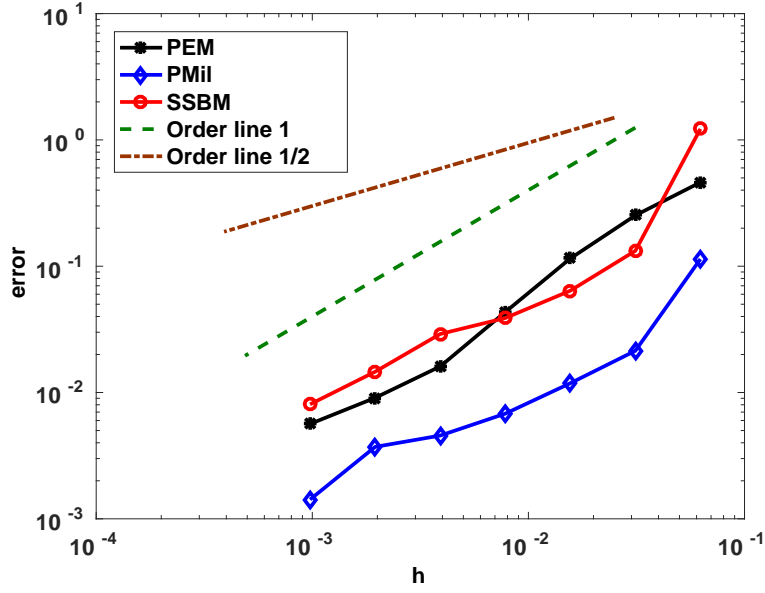}
\caption{Strong convergence errors for the approximation of the double well 
dynamics. Parameter values $\sigma=1$ 
 and $X_0=2$.  } 
\label{fig2}
\end{figure}

% \raphael{ I suspect the important thing here is the violation of the moment
% bound. We should check if the moment bound for the PEM methods from the other
% paper is violated, too. If it still holds that paper may explain why we still
% have convergence but with a lower order.}
\begin{table}[H]
  \caption{Strong convergence errors for the approximation of the double well
  dynamics. Parameter values $\sigma=0.3$ and $X_0=2$.}
  \label{tab:ginz1}
  \begin{tabular}{p{0.9cm}p{1.2cm}p{0.9cm}p{1cm}p{1.2cm}p{0.9cm}p{1cm}p{1.2cm}p{0.9cm}}
       &  PEM  &  & & PMil &  & &  SSBM  &  
\\   \noalign{\smallskip}\hline\noalign{\smallskip}
  $h$   &   error  &  EOC &  \#  &   error   &  EOC  & \# & error  & EOC
\\   \noalign{\smallskip}\hline\noalign{\smallskip}
       $2^{-4}$ &    0.4611  &    &   233635 &  0.1139 &  &  331721 & 1.2222\\ 
     $2^{-5}$ &       0.2525  &          0.87  &      877 &       0.0214 &     
2.41  &      516 &       0.1324  &         3.21\\
     $2^{-6}$ &       0.1151  &          1.13  &      212 &       0.0118 &      
    0.85  &      168 &       0.0640  &         1.05\\
     $2^{-7}$ &       0.0433  &          1.41  &       58 &       0.0068 &      
    0.80  &       63 &       0.0389  &         0.72\\
     $2^{-8}$ &       0.0161  &          1.43  &       21 &       0.0046 &      
    0.58  &       24 &       0.0291  &         0.42\\
     $2^{-9}$ &       0.0091  &          0.83  &       11 &       0.0037 &      
    0.30  &        9 &       0.0146  &         0.99\\
     $2^{-10}$ &       0.0056  &          0.69  &        1 &       0.0014 &     
     1.38  &        2 &       0.0081  &         0.86
  \end{tabular}
\end{table}

\begin{figure}[p!]
    \subfigure[The case $\sigma=0.3$.]
{\includegraphics[width=0.80\textwidth]{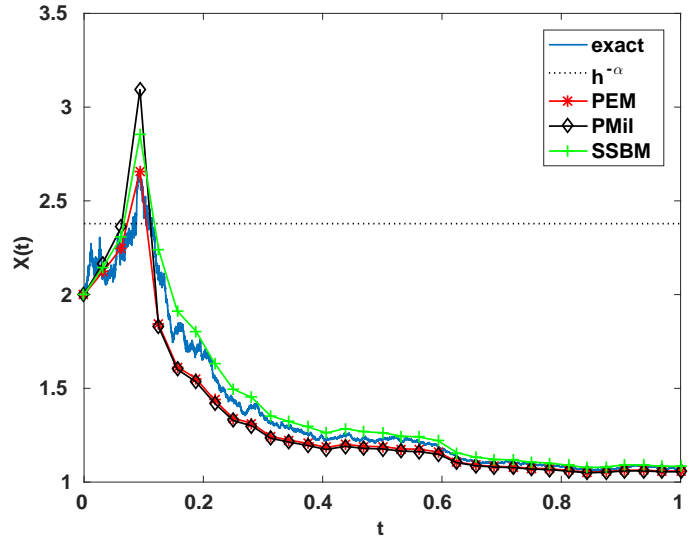}}
    \subfigure[The case $\sigma=1$.]
{\includegraphics[width=0.80\textwidth]{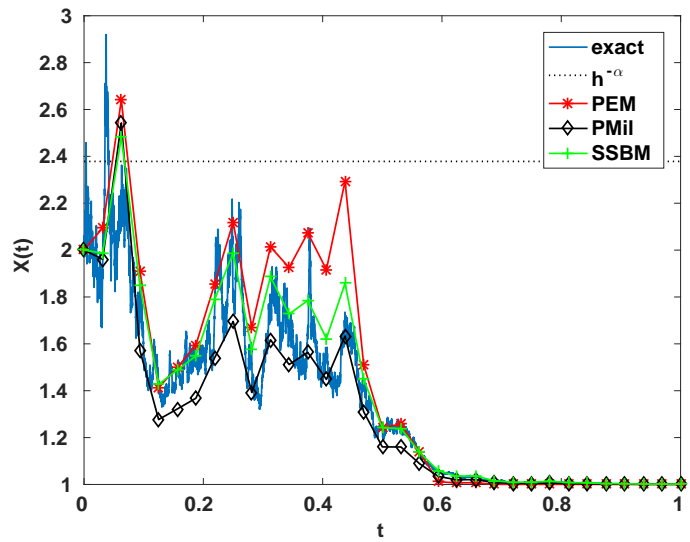}}
\caption{Single trajectories for the methods PEM, PMil, and SSBM with step size
$h= 2^{-5}$ versus reference solution (exact) with parameter values 
$\alpha=\frac{1}{4}$, $X_0=2,$ and $T=1$.}
\label{fig2a}
\end{figure}
\clearpage

Figure~\ref{fig2} and Table~\ref{tab:ginz1} show the results of the strong 
error of convergence when conditions \eqref{eq3:onesided} and  
\eqref{eq:growthcond} in Assumption~\ref{as:fg} are violated by choosing 
$\sigma=1$. The 
estimate of the errors are based on the Monte Carlo simulation with the same 
number $2 \cdot10^6$  of samples as above. And as in the first experiment, 
confidence intervals are two orders of magnitude smaller than the values 
themselves. Therefore, we believe that the slightly irregular behavior of 
the convergence errors is not due to a too small  number of samples. 
Rather we suspect that  violation of the 
convergence conditions influences the expected order of convergence, see 
the numerical EOC values in Table~\ref{tab:ginz1}. This effect certainly 
deserves further investigation. For an illustration we include two runs
for parameter values $\sigma=0.3$ and $\sigma=1$ in Figure~\ref{fig2a}.

\begin{figure}[h]
\includegraphics[width=11cm]{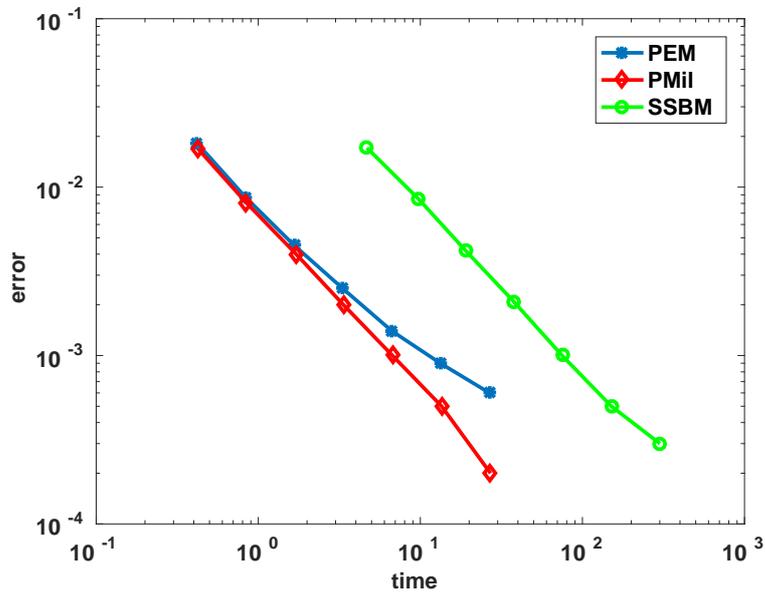}
\caption{CPU times versus  $L^2$-errors of the PEM, PMil, and SSBM methods for 
the double well dynamics.} 
\label{fig3}
\end{figure}

For a fair comparison of computational costs we compiled Figure \ref{fig3} 
which shows the $L^2$-error versus computing times (measured by tic,toc in 
MATLAB). One clearly observes that
PMil and PEM outperform the split-step backward Milstein method SSBM.
Moreover, PMil has a slight advantage over PEM when high accuracy is required.

\subsection*{A stochastic oscillator with commutative noise}
%
%
%\noindent
%{\bf Stochastic Hopf bifurcation with commutative noise}

\begin{figure}[p!]
    \subfigure[Single trajectories of exact solution and PMil scheme.]
{\includegraphics[width=0.80\textwidth]{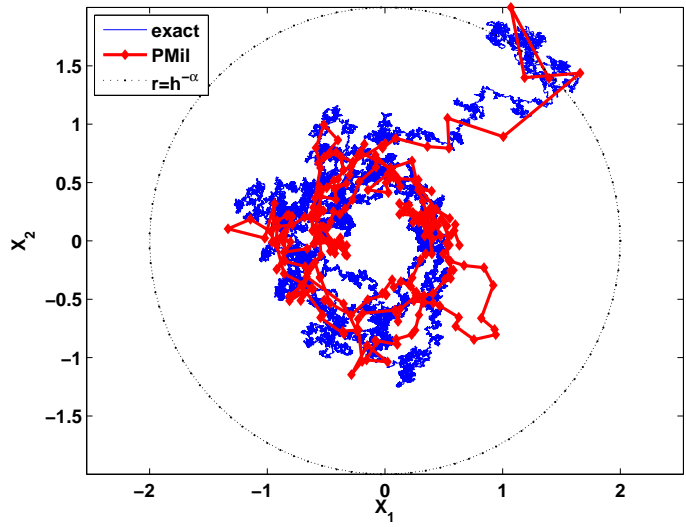}}
    \subfigure[Zoom in on the first few steps of the trajectory from (a).]
{\psfrag{ax}{$X_h(t_0)$}
\psfrag{xb}{$X_h(t_1)$}
\psfrag{cx}{$\overline{X}_h(t_2)$}
\psfrag{dx}{$X_h(t_2)$}
\psfrag{ex}{$X_h(t_3)$}
\psfrag{fx}{$\overline{X}_h(t_4)$}
\psfrag{gx}{$X_h(t_4)$}
\includegraphics[width=0.80\textwidth]{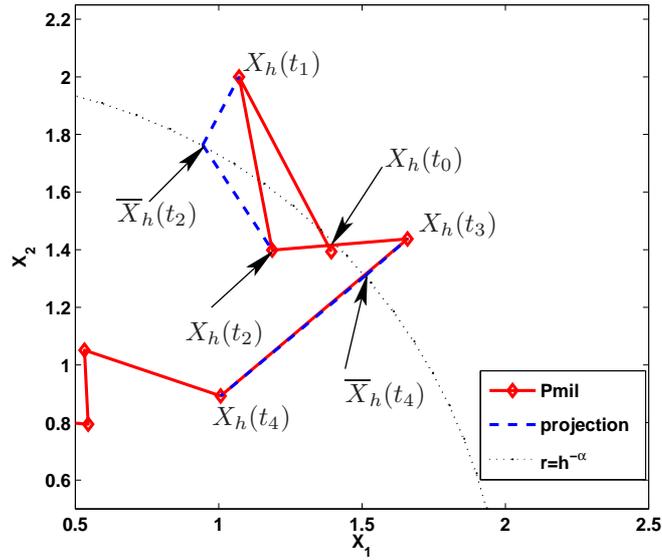}}
\caption{Single trajectory of PMil scheme with step size $h=2^{-4}$ for the 
stochastic oscillator dynamics. Two projected intermediate steps  are 
indicated by dashed lines. Parameter values 
$\mu=0.4$, $\sigma_1=0.5$, $\sigma_2=0.6$, $\theta=1$, $r_0=1.97$ and 
$\varphi_0=\frac{\pi}{4}$. }
\label{fig4}
\end{figure}

Next, we consider a system in polar coordinates with diagonal noise of the 
following
form 
\begin{equation}
  \label{eq:polcor}
  \begin{pmatrix}
    \diff r(t) \\
    \diff \varphi(t) 
  \end{pmatrix}
  = 
  \begin{pmatrix} 
    r f_1(r^2) \\ 
    f_2(\varphi)
  \end{pmatrix} 
  \diff t +
  \begin{pmatrix} r g_1(r^2) \diff W^1(t) \\
    g_2(\varphi) \diff W^2(t)
  \end{pmatrix},
\end{equation}
where $f_1,f_2$ are smooth functions on $[0,\infty)$ and
$g_1,g_2$ are smooth $2 \pi$-periodic functions on $\R$.
We transform to Euclidean coordinates by applying It\=o's formula to
$x(r,\varphi)=(r \cos(\varphi),r \sin(\varphi))$, 
\begin{equation}
  \label{eq:cartgeneral}
  \begin{aligned}
    \diff x(t) &= \left[ 
    \begin{pmatrix} 
      x_1 & - x_2 \\
      x_2 & x_1 
    \end{pmatrix}
    \begin{pmatrix}
      f_1(r^2) \\ 
      f_2(\varphi) 
    \end{pmatrix}
    -\frac{1}{2} (g_2(\varphi))^2 
    \begin{pmatrix} x_1 \\
      x_2
    \end{pmatrix}
    \right] \diff t\\
    &\quad + g_1(r^2) 
    \begin{pmatrix}
      x_1 \\
      x_2 
    \end{pmatrix} 
    \diff W^1(t)
    + g_2(\varphi) 
    \begin{pmatrix} 
      - x_2 \\ 
      x_1 
    \end{pmatrix}
    \diff W^2(t).
\end{aligned}
\end{equation}
Here we set $r^2 = x_1^2 + x_2^2$ and replace $g_j(\varphi),j=1,2$
by $g_j(\mathrm{arg}(x_1 +ix_2))$ with the argument
$\mathrm{arg}(r e^{i \varphi})=\varphi$ taken from $(-\pi,\pi]$, for example.
 In the following computation we treat the special case
\begin{equation} 
  \label{eq:setparamHopf}
  f_1(r^2)=\mu - r^2, \quad f_2(\varphi) = \theta, \quad
  g_1(r^2)=\sigma_1,  \quad g_2(\varphi) = \sigma_2
\end{equation}
with parameters $\mu,\theta,\sigma_1,\sigma_2 \in \R$ still to be chosen. This
is a generalization of a system studied in \cite{gorini2015}. 
When the parameter $\mu$ varies, this may be considered as a model problem for
stochastic Hopf bifurcation (cf.  \cite[Chap.~9.4.2]{arnold1998}).

In this case the system \eqref{eq:cartgeneral} with
initial condition $x(0) =(r_0\cos(\varphi_0),r_0 \sin(\varphi_0))$
can be solved explicitly via \eqref{eq:polcor}, since the radial
equation is a stochastic Ginzburg Landau equation while
the angular equation can be directly integrated. We obtain from
\cite[Chap.~4.4]{kloeden1999} 
\begin{equation}
  \label{eq:exactHopf}
  \begin{aligned}
    r(t) &= r_0 \exp\big( (\mu-\frac{\sigma_1^2}{2})t+ \sigma_1 W^1(t)\big)\\
    &\qquad \times 
    \Big(1+2 r_0^2\int_0^t\exp\big((2\mu-\sigma_1^2)s+ 2 \sigma_1
    W^1(s) \big) \diff s \Big)^{-1/2}, \\ 
    \varphi(t) &= \varphi_0 + \theta t + \sigma_2 W^2(t).
  \end{aligned}
\end{equation}
Since $g^1,g^2$ are linear and $f$ is cubic with a uniform upper Lipschitz
bound, we find that Assumption \ref{as:fg} is satisfied with $q=3$.
Moreover, the system has commutative noise \cite[Chap.~10.3]{kloeden1999},
since 
\begin{equation*}
  g^{1,2}(x)=g^{2,1}(x)= \sigma_1 \sigma_2 
  \begin{pmatrix}
    -x_2 \\ 
    x_1 
  \end{pmatrix}.
\end{equation*}
As in \cite[Chap.10 (3.16)]{kloeden1999} the double sum in \eqref{eq:PMildef},
\eqref{eq:SSBMdef} then takes the explicit form
\begin{equation*}
  \frac{1}{2}\left[\sigma_1^2( ( I_{(1)}^{t_{i-1},t_i})^2 - h_i) -
  \sigma_2^2( ( I_{(2)}^{t_{i-1},t_i} )^2 - h_i) \right] Y 
  + \sigma_1 \sigma_2 I_{(1)}^{t_{i-1},t_i} I_{(2)}^{t_{i-1},t_i}
  \begin{pmatrix}
    -Y_2 \\
    Y_1 
  \end{pmatrix},
\end{equation*}
where $Y=\overline{X}_h^{\mathrm{PMil}}(t_i)$ and 
$Y=\overline{X}_h^{\mathrm{SSBM}}(t_i)$, respectively.

Figure~\ref{fig4} (a) shows the simulation of a single path generated by the exact 
solution and the projected Milstein method with equidistant step size 
$h=2^{-4}$ and  parameters $\alpha=\frac{1}{4}, \mu=0.4, \sigma_1=0.5, 
\sigma_2=0.6, \theta=1$, $r_0=1.97$, and $\varphi_0=\frac{\pi}{4}$. The initial 
value is $X_0=(1.39,1.39)$. Further, we use a highly accurate approximation of the integral in 
\eqref{eq:exactHopf} by a Riemann sum with step size $\Delta 
t=2^{-18}$.

As already mentioned in  \cite{beyn2015} we are interested in trajectories of 
the PMil scheme which do not coincide with trajectories generated by  the 
standard Milstein method. Such a case is shown in Figure~\ref{fig4} (a) where
the exact trajectory and the PMil-trajectory are displayed.
Figure~\ref{fig4} (b) shows a close-up of the 
projected Milstein scheme near the circle of radius $h^{-\alpha}=2$.
In the first and in the third step the trajectory leaves the ball,
creating in the next step the intermediate values 
$\overline{X}_h^{\mathrm{PMil}}(t_2)$ and $\overline{X}_h^{\mathrm{PMil}}(t_4)$, which have been
connected by dashed lines to their predecessor and their successor.
 Obviously, this event occurs more often when the 
starting point is close to the  circle  and  the values of $\sigma_1$ 
and $\sigma_2$ are large.

\begin{figure}[H]
\includegraphics[width=0.80\textwidth]{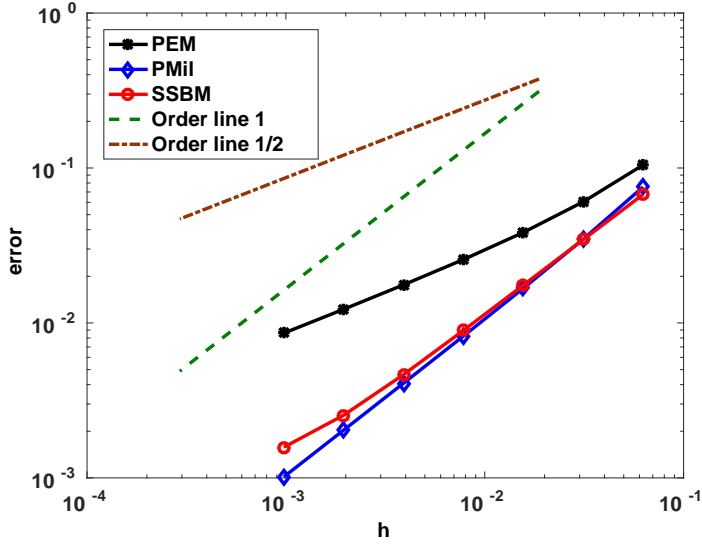}
\caption{Strong convergence errors for the approximation of the stochastic 
oscillator. Parameter values $\mu=0.4$, $\sigma_1=0.5$, $\sigma_2=0.6$, 
$\theta=1$, $r_0=1.97$, and $\varphi_0=\frac{\pi}{4}$.} 
\label{fig5}
\end{figure}

\begin{table}[b]
  \caption{Strong convergence errors for the approximation of the stochastic 
oscillator. Parameter values $\mu=0.4$, $\sigma_1=0.5$, $\sigma_2=0.6$, 
$\theta=1$, $r_0=1.97$, and $\varphi_0=\frac{\pi}{4}$.}
\label{tab:hopf} 
\begin{tabular}{p{0.9cm}p{1.2cm}p{0.9cm}p{1cm}p{1.2cm}p{0.9cm}p{1cm}p{1.2cm}p{ 
0.9cm}}   &  PEM  &  &  & PMil   &  &  &  SSBM  &                  
\\   \noalign{\smallskip}\hline\noalign{\smallskip}
  $h$   & error  &  EOC &  \#  & error &  EOC  & \# & error  & EOC
\\   \noalign{\smallskip}\hline\noalign{\smallskip}
$2^{-4}$ &   0.1045  &    &     3082 &   0.07540 &    &     4279 & 0.06741  \\ 
$2^{-5}$ & 0.06045 & 0.79  & 1 &  0.03468 & 1.12  &  0 &  0.03445  & 0.97\\
$2^{-6}$ & 0.03838  & 0.66  & 0 & 0.01673 & 1.05  &  0 & 0.01753 &  0.97\\
$2^{-7}$ & 0.02566  & 0.58  &  0 & 0.00823 & 1.02  & 0 & 0.00894 & 0.97\\
$2^{-8}$ & 0.01762  & 0.54  &  0 & 0.00408 & 1.01  &  0 & 0.00464  &  0.95\\
$2^{-9}$ & 0.01226  & 0.52  &  0 & 0.00204 & 1.01  & 0 & 0.00254  &   0.87\\
$2^{-10}$ & 0.00860  & 0.51  & 0 & 0.00102 & 1.00  &  0 & 0.00158 &  0.69
\end{tabular}
\end{table}

Figure~\ref{fig5} and Table~\ref{tab:hopf} show the estimated  strong error of 
convergence for the PEM scheme, the PMil method, and the SSBM scheme.
Nonlinear equations in the scheme SSBM are solved by  Newton's method with 
three iteration steps. The parameters and the initial value are as in
Figure~\ref{fig4}. 

The estimates of errors, given by \eqref{eq:error} at the endpoint $T=1$, 
with seven different step sizes $h=2^k\Delta t, k=8,\dots, 14$ are 
again based on Monte Carlo simulations with $2\cdot10^6$ samples. As above the
associated confidence intervals are two orders of magnitude smaller than the
estimated errors.

The numerical results in  Figure~\ref{fig5} and 
in Table~\ref{tab:hopf} confirm the theoretical
 orders of convergence, though with some loss towards smaller step-sizes for
 the SSBM-method.
As in our first example we provide a diagram of error versus computing time in 
Figure \ref{fig6}.
This time PEM has a slight advantage for very rough accuracy, but is worse
than PMil and SSBM for higher accuracy. PMil always wins against SSBM.
\begin{figure}[H]
\includegraphics[width=11cm]{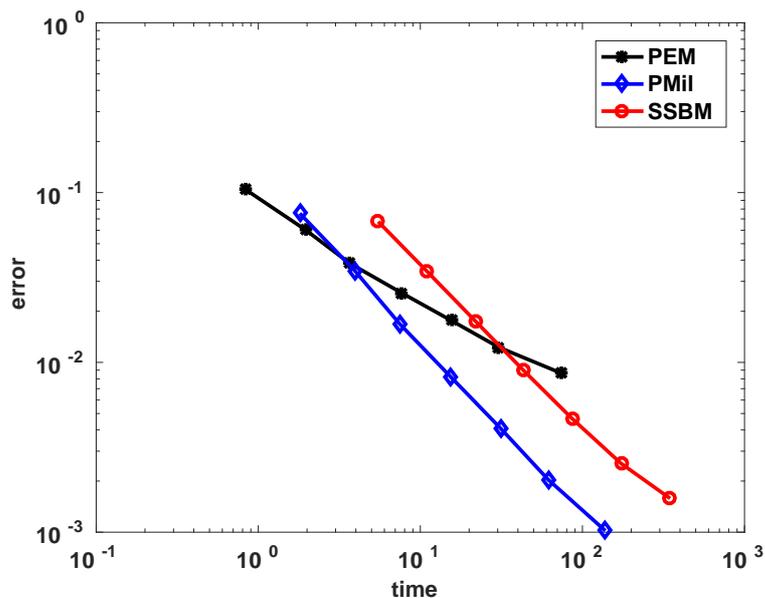}
\caption{CPU times versus  $L^2$-errors of the PEM, PMil, and SSBM methods for 
the stochastic oscillator.} 
\label{fig6}
\end{figure}

\subsection*{Acknowledgement}

The authors would like to thank Martin Steinborn for bringing several typos to
our attention. In addition, the first two authors are grateful for financial
support by the DFG-funded CRC 701 'Spectral Structures and Topological Methods
in Mathematics'. Further, this research was carried out by the third
named author in the framework of {\sc Matheon}, project A25, supported by
Einstein Foundation Berlin.

%\bibliographystyle{plain}
%\bibliography{../lit}

\def\cprime{$'$} \def\polhk#1{\setbox0=\hbox{#1}{\ooalign{\hidewidth
  \lower1.5ex\hbox{`}\hidewidth\crcr\unhbox0}}}

\end{document}